\documentclass{amsart}
\usepackage{amssymb,amsmath,amsthm, mathabx, hyperref}
\usepackage{tikz}
\usepackage{cancel}
\usepackage{mathtools,float}
\usepackage{enumerate}
\usepackage{todonotes}
\RequirePackage{xargs}

\newcommand{\shadetheboxes}[1]{
    \foreach \x/\y in {#1}
    \fill[pattern color = black!75, pattern=north east lines] (\x,\y) rectangle +(1,1);
}

\newcommand{\drawthegrid}[1]{
    \draw (0.01,0.01) grid (#1+0.99,#1+0.99);
}

\newcommand{\drawverticallines}[3]{
    \foreach \x in {#2}
    \draw[line width=#3, line cap=round] (\x+0.01,0.01) -- (\x+0.01,#1+0.99);
}

\newcommand{\drawhorizontallines}[3]{
    \foreach \y in {#2}
    \draw[line width=#3, line cap=round] (0.01,\y+0.01) -- (#1+0.99,\y+0.01);
}

\newcommand{\drawclpattern}[2]{
	\foreach[count=\x] \y in {#1}
	{
		\filldraw (\x,\y) circle (#2 pt);
	}
}

\newcommand{\drawspecialbox}[1]{
    \foreach \x/\y/\z/\w/\A in {#1}
    {
        \fill[color = white!100, opacity=1, rounded corners = 1.5pt] (\x+0.125,\y+0.125) rectangle (\z-0.125,\w-0.125);
        \draw[color = black, rounded corners = 1.5pt] (\x+0.125,\y+0.125) rectangle (\z-0.125,\w-0.125);
        \fill[black] (\x/2+\z/2,\y/2+\w/2) node {\A};
    }
}

\newcommand{\drawspecialboxlarge}[1]{
    \foreach \x/\y/\z/\w/\A in {#1}
    {
        \fill[color = white!100, opacity=1, rounded corners = 1.5pt] (\x+0.125,\y+0.125) rectangle (\z-0.125,\w-0.125);
        \draw[color = black, rounded corners = 1.5pt] (\x+0.125,\y+0.125) rectangle (\z-0.125,\w-0.125);
        \fill[black] (\x/2+\z/2,\y/2+\w/2) node {\Large \A};
    }
}

\newcommand{\drawsolidshadedbox}[1]{
    \foreach \x/\y/\z/\w/\A in {#1}
    {
        \fill[color = gray!50, opacity=1, rounded corners=1.5pt] (\x+0.125,\y+0.125) rectangle (\z-0.125,\w-0.125);
        \draw[color = black, rounded corners=1.5pt] (\x+0.125,\y+0.125) rectangle (\z-0.125,\w-0.125);
        \fill[black] (\x/2+\z/2,\y/2+\w/2) node {\A};
    }
}

\newcommand{\drawlabels}[1]{
	\foreach \x/\y/\lab in {#1}
	{
		\draw (\x + 0.5,\y + 0.5) node {\lab};
	}
}

\newcommandx{\patt}[9][4={},5={},6={},7={},8={},9=4]
{
	\scalebox{#1}
	{
		\begin{tikzpicture}[baseline=(current bounding box.center)]
			\useasboundingbox (0.0,-.3) rectangle (#2+1,#2+1);
			\shadetheboxes{#4}
			\draw (0.01,0.01) grid (#2+1-0.01,#2+1-0.01);

			\drawsolidshadedbox{#6}
			\drawspecialbox{#7}
			\drawspecialboxlarge{#5}
			\drawclpattern{#3}{#9}
			\drawlabels{#8}
		\end{tikzpicture}
	}
}

\newcommandx{\cpatt}[8][4={},5={},6={},7={},8={}]
{
	\scalebox{#1}
	{
		\begin{tikzpicture}[baseline=(current bounding box.center)]
			\useasboundingbox (0.0,-.3) rectangle (#2+1,#2+1);
			\shadetheboxes{#4}
			\draw (0.01,0.01) grid (#2+1-0.01,#2+1-0.01);

			\drawsolidshadedbox{#6}
			\drawspecialbox{#7}
			\drawspecialboxlarge{#5}
			\drawclpattern{#3}{4}

			\foreach \x/\y in {#8}
			{
				\draw[line width=1] (\x,\y) circle (7 pt);
			}
		\end{tikzpicture}
	}
}

\newcommandx{\metapatt}[8][6={},7={},8={}]
{
    \scalebox{#1}
    {
        \begin{tikzpicture}[baseline=(current bounding box.center)]
					\foreach \width/\height in {#2}
					{
						\useasboundingbox (0.0,-.3) rectangle (\width+1,\height+1);
            \shadetheboxes{#6}

            \foreach \pos/\type in {#4}
            {
                \ifthenelse{\equal{\type}{v}}
                {
                    \drawverticallines{\height}{\pos}{1.7pt}
                }
                {
								    \ifthenelse{\equal{\type}{d}}
                    {
                      \draw[densely dashed] (\pos,0) -- (\pos,\height+1);
                    }
										{
											\drawhorizontallines{\width}{\pos}{1.7pt}
										}
                }
            }

            \foreach \pos/\type in {#3}
            {
                \ifthenelse{\equal{\type}{v}}
                {
                    \drawverticallines{\height}{\pos}{0.6pt}
                }
                {
										\drawhorizontallines{\width}{\pos}{0.6pt}
                }
            }

            \drawsolidshadedbox{#8}
            \drawspecialbox{#7}

            \foreach \x/\y/\type in {#5}
            {
                \ifthenelse{\equal{\type}{a}}
                {
                    \draw (\x,\y) circle (6pt);
                    \filldraw (\x,\y) circle (3pt);
                }
                {
                    \filldraw (\x,\y) circle (4pt);
                }
            }
					}
        \end{tikzpicture}
    }
}

\newcommandx{\dpatt}[9][6={},7={},8={},9={}]
{
    \scalebox{#1}
    {
        \begin{tikzpicture}[baseline=(current bounding box.center)]
					\foreach \width/\height in {#2}
					{
						\useasboundingbox (0.0,-.3) rectangle (\width+1,\height+1);
            \shadetheboxes{#6}

            \foreach \pos/\type in {#4}
            {
                \ifthenelse{\equal{\type}{v}}
                {
                    \drawverticallines{\height}{\pos}{1.7pt}
                }
                {
								    \ifthenelse{\equal{\type}{d}}
                    {
                      \draw[densely dashed] (\pos,0) -- (\pos,\height+1);
                    }
										{
											\drawhorizontallines{\width}{\pos}{1.7pt}
										}
                }
            }

            \foreach \pos/\type in {#3}
            {
                \ifthenelse{\equal{\type}{v}}
                {
                    \drawverticallines{\height}{\pos}{0.6pt}
                }
                {
										\drawhorizontallines{\width}{\pos}{0.6pt}
                }
            }

            \drawsolidshadedbox{#8}
            \drawspecialbox{#7}

            \foreach \x/\y/\type in {#5}
            {
                \ifthenelse{\equal{\type}{a}}
                {
                    \draw9 (\x,\y) circle (6pt);
                    \filldraw (\x,\y) circle (3pt);
                }
                {
                    \filldraw (\x,\y) circle (4pt);
                }
            }

						\drawlabels{#9}
					}
        \end{tikzpicture}
    }
}

\newcommand{\textmpattern}[4]{										
    \scalebox{#1}
    {
      \begin{tikzpicture}[baseline=(current bounding box.center)]
      	\useasboundingbox (0.0,0) rectangle (#2+1,#2+0);
        \shadetheboxes{#4}

        \drawthegrid{#2}

        \drawclpattern{#3}{6}

      \end{tikzpicture}
    }
}

\newcommandx{\shpatt}[7][4={},5={},6={},7={}]
{
	\scalebox{#1}
	{
		\begin{tikzpicture}[baseline=(current bounding box.center)]
			\useasboundingbox (0.0,-.3) rectangle (#2+1,#2+1);
			\shadetheboxes{#4}
			\draw (0.01,0.01) grid (#2+1-0.01,#2+1-0.01);

			\drawclpattern{#3}{4}

			\foreach \x/\y in {#5}
			{
				\draw[line width=1] (\x,\y) circle (7 pt);
			}

            \foreach \x/\y in {#6}
            {
                \draw[densely dashed, line width=1.7pt] (\x,0) -- (\x,#2+1);
                \draw[densely dashed, line width=1.7pt] (0,\y) -- (#2+1,\y);
            }

            \foreach \xa/\ya/\xb/\yb in {#7}
            {
                \draw[->, line width=1.7pt] (\xa,\ya) -- (\xb-0.12,\yb-0.12);
            }
		\end{tikzpicture}
	}
}

\newcommandx{\shpattb}[8][4={},5={},6={},7={},8={}]
{
	\scalebox{#1}
	{
		\begin{tikzpicture}[baseline=(current bounding box.center)]
			\useasboundingbox (0.0,-.3) rectangle (#2+1,#2+1);
			\shadetheboxes{#4}
			\draw (0.01,0.01) grid (#2+1-0.01,#2+1-0.01);

			\drawclpattern{#3}{4}

			\foreach \x/\y in {#5}
			{
				\draw[line width=1] (\x,\y) circle (7 pt);
			}

            \foreach \x/\y in {#6}
            {
                \draw[densely dashed, line width=1.7pt] (\x,0) -- (\x,#2+1);
                \draw[densely dashed, line width=1.7pt] (0,\y) -- (#2+1,\y);
            }

            \foreach \xa/\ya/\xb/\yb in {#7}
            {
                \draw[->, line width=1.7pt] (\xa,\ya) -- (\xb-0.12,\yb-0.12);
            }
            \foreach \xa/\ya/\xb/\yb in {#8}
            {
            	\draw[line width=1.5pt] (\xa+0.1,\ya+0.1) rectangle (\xb-0.1,\yb-0.1);
            }
		\end{tikzpicture}
	}
}

\usepackage{stackengine}
\usepackage[T1]{fontenc}
\usepackage[utf8]{inputenc}
\usepackage[english]{babel}
\usepackage{csquotes, xpatch}

\usetikzlibrary{arrows,shapes,automata,backgrounds,decorations,petri,
                positioning,patterns}
\colorlet{lightgray}{black!15}

\newcommand{\boks}[1]{{\lcorners{#1}\rcorners}}
\newcommand{\point}[1]{{\boks{#1}}}
\newcommand{\pleft}[1]{{\boks{#1}\leftarrow}}
\newcommand{\pright}[1]{{\boks{#1}\rightarrow}}
\newcommand{\pup}[1]{{\boks{#1}\uparrow}}
\newcommand{\pdown}[1]{{\boks{#1}\downarrow}}
\newcommand{\pall}[1]{{\boks{#1}\star}}

\newcommand{\av}[1]{{{\mathrm{Av}}(#1)}}
\newcommand{\co}[1]{{\mathrm{Co}(#1)}}
\newcommand{\avn}[2]{{{\mathrm{Av}_{#2}}(#1)}}

\newcommand{\occ}[2]{{\mathrm{occ}_{#1}(#2)}}
\newcommand{\strength}{\mathrm{strength}}

\newcommand{\tsa}{{\textsc{SA}}}
\newcommand{\symm}{S}

\theoremstyle{definition}
\newtheorem{theorem}{Theorem}[section]
\newtheorem{remark}[theorem]{Remark}
\newtheorem{corollary}[theorem]{Corollary}
\newtheorem{proposition}[theorem]{Proposition}
\newtheorem{definition}[theorem]{Definition}
\newtheorem{lemma}[theorem]{Lemma}
\newtheorem{conjecture}[theorem]{Conjecture}
\newtheorem{problem}[theorem]{Problem}
\numberwithin{equation}{section}
\numberwithin{figure}{section}

\newtheorem{example}[theorem]{Example}

\usepackage{algorithm}
\usepackage{algpseudocode}
\newcommand{\To}{\textbf{to}}

\algdef{SE}[DOWHILE]{Do}{doWhile}{\algorithmicdo}[1]{\algorithmicwhile\ #1}%
\algblockdefx[MyBlock]{Begin}{End}%
{\textbf{begin}}%
{\textbf{end}}

\usetikzlibrary{calc,decorations,decorations.pathreplacing,decorations.shapes,decorations.markings,shapes,patterns,positioning}
\tikzset{invisible/.style={minimum width=0mm,inner sep=0mm,outer sep=0mm}}
\tikzset{tiling/.style={invisible, anchor=north west}}
\tikzset{ptnode/.style={rounded corners=15, black!40, thick, dashed}}

\pgfmathsetmacro{\patttablescale}{1.05}
\pgfmathsetmacro{\pattdispscale}{0.475}
\pgfmathsetmacro{\patttextscale}{0.175}

\title{Algorithmic coincidence classification of mesh patterns}

\author[C.~Bean]{Christian Bean$^{\star}$}
\address{School of Computer Science, Reykjavik University, Reykjavik, Iceland}
\email{henningu@ru.is}
\author[B.~Gudmundsson]{Bjarki Gudmundsson$^{\star}$}
\address{School of Computer Science, Reykjavik University, Reykjavik, Iceland}
\email{henningu@ru.is}
\author[T.~Magnusson]{Tomas Ken Magnusson$^{\star}$}
\address{School of Computer Science, Reykjavik University, Reykjavik, Iceland}
\email{henningu@ru.is}
\author[H.~Ulfarsson]{Henning Ulfarsson$^{\star}$}
\address{Department of Computer Science, Reykjavik University, Reykjavik, Iceland}
\email{henningu@ru.is}

\thanks{$^{\star}$ Research partially
  supported by grant 141761-051 from the Icelandic Research Fund.}


\begin{document}

\begin{abstract}
We review and extend previous results on coincidence of mesh patterns.
We introduce the notion of a force on a permutation pattern and apply it
to the coincidence classification of mesh patterns, completing the classification
up to size three.
We also show that this concept can be used to enumerate classical
permutation classes.

\end{abstract}

\maketitle
\thispagestyle{empty}

Permutation patterns have been studied since the beginning of the
20\textsuperscript{th} century, starting with MacMahon~\cite{macmahon2001combinatory},
who considered the union of two decreasing sequences of points. Interest in
modern day study was sparked by Simion and Schmidt~\cite{simion1985restricted}. Classical
permutation patterns have been generalized to \emph{vincular}
patterns by Babson and Steingr\'{i}msson~\cite{babson2000generalized}, to \emph{bivincular}
patterns by Bousquet-M\'{e}lou et al.~\cite{bousquet20102} and to \emph{barred} patterns
by West~\cite{west1993sorting}.
The focus of this paper will be on \emph{mesh} patterns, introduced by
Br\"{a}nd\'{e}n and Claesson~\cite{branden2011mesh}, and subsume the prior generalizations.
In particular, we study the classification of
these patterns based on \emph{coincidence}, an equivalence relation derived from
the avoidance sets of the patterns. A related subject is the classification of
patterns in terms of \emph{Wilf-equivalence} where the size of the avoidance
sets determines the relation.
This line of inquiry for mesh patterns was started by
Hilmarsson et al.~\cite{shadinglemma}, who provided sufficient conditions
for the coincidence of mesh patterns with the so-called Shading lemma. This was
further generalized by Claesson, Tenner, and Ulfarsson~\cite{simshadinglemma} in the Simultaneous Shading
lemma. The relationship between the avoidance sets of mesh patterns and
classical patterns has been studied by
Tenner~\cite{tenner2013coincidental,tenner2013mesh}, who determined which
mesh patterns are coincident with classical patterns.

%
%

We will review these earlier results, and extend them using the notion of
a \emph{force} on a permutation pattern. This will culminate in the Shading Algorithm,
which is powerful enough to coincidence classify the set of mesh patterns of size $3$,
except one case which we do by hand. We show how knowing the coincidence of mesh patterns
can be used to enumerate permutation sets avoiding a classical pattern. Furthermore,
we show how the concept of a force can be applied directly to that problem.

\section{Preliminaries}
A \emph{permutation} on a set $A$ is a bijection from $A$ to itself.
In this
paper we always have $A = [1,n] = \{1,\ldots,n\}$ for some non-negative integer
$n$. We denote the set of all permutations of size $n$ as $\symm_n$ and we
write $\pi \in \symm_n$ as the word $\pi(1)\pi(2) \cdots \pi(n)$.
Two sequences of integers $a_1a_2 \cdots a_k$ and $b_1b_2 \cdots b_k$ are
\emph{order isomorphic} if $a_i < a_j$ if and only if $b_i < b_j$ for all
$i,j \in [1,k]$. The central definition in the study of permutation patterns
is the following.

\begin{definition}
  A permutation $\pi \in \symm_n$ \emph{contains} a
  permutation $p \in \symm_k$ if there exists a sequence of indices $1\leq i_1
    < i_2 < \cdots < i_k \leq n$ such that the subsequence
  $c = \pi(i_1) \pi(i_2) \cdots \pi(i_k)$ of $\pi$ is order
  isomorphic to $p$. The subsequence $c$ is called an \emph{occurrence} of
  $p$ in $\pi$. If such a subsequence does not exist then $\pi$ \emph{avoids}
  $p$. In this context $p$ is called a \emph{classical} (\emph{permutation})
  \emph{pattern}.
\end{definition}

A permutation $\pi \in \symm_n$ can be represented graphically by plotting the
points
$
\{(i,\pi(i)) \,\vert\, i \in [1,n] \}
$
on a grid such that the final figure resembles a \emph{mesh}. The elements $(i,
\pi(i))$ are referred to as the \emph{points} of the permutation $\pi$. An
example is illustrated in the first subfigure of Figure~\ref{fig:flattenex}.

An occurrence of a pattern in a permutation can be viewed in the
graphical representation as scaling the indices and the values of the pattern
to coincide with points of the permutation, while maintaining the same relative ordering. An example is
depicted in Figure~\ref{fig:flattenex} with the graph of the permutation $42135$
and one of its subsequences.
\begin{figure}[ht]
  \centering
  \raisebox{0.6ex}{%
    \patt{\pattdispscale}{3}{2,1,3}{}
  }
  \quad\quad
  \raisebox{0.6ex}{
    \begin{tikzpicture}[baseline=(current bounding box.center), scale=\pattdispscale]
      \useasboundingbox (0.0,-0.1) rectangle (5+1.4,5+1.1);
      \foreach [count=\x] \y in {4,2,1,3,5}
      \filldraw (\x,\y) circle (4pt);
      \foreach \x/\y in {1/4,3/1,5/5}
      \draw[thick,red!70!black] (\x,\y) circle (7pt);
      \draw[very thin] (1,0.01) -- (1,5.99);
      \draw[very thin] (3,0.01) -- (3,5.99);
      \draw[very thin] (5,0.01) -- (5,5.99);
      \draw[very thin] (0.01,1) -- (5.99,1);
      \draw[very thin] (0.01,4) -- (5.99,4);
      \draw[very thin] (0.01,5) -- (5.99,5);
      \draw[very thin] (3,0.01) -- (3,5.99);
      \draw[very thin] (5,0.01) -- (5,5.99);

      \draw[densely dotted, line width=0.6pt] (2,0.01) -- (2,5.99);
      \draw[densely dotted, line width=0.6pt] (4,0.01) -- (4,5.99);
      \draw[densely dotted, line width=0.6pt] (0.01,2) -- (5.99,2);
      \draw[densely dotted, line width=0.6pt] (0.01,3) -- (5.99,3);
    \end{tikzpicture}}
    \quad\quad
    \raisebox{0.6ex}{
      \begin{tikzpicture}[baseline=(current bounding box.center), scale=\pattdispscale]
        \useasboundingbox (0.0,-0.1) rectangle (3+1.4,3+1.1);
        \draw[very thin] (0.01,0.01) grid (3+0.99,3+0.99);
        \foreach [count=\x] \y in {2,1,3}
        \filldraw (\x,\y) circle (4pt);
        \foreach \x/\y in {1/2,2/1,3/3}
        \draw[thick,red!70!black] (\x,\y) circle (7pt);
        \draw[densely dotted,line width=0.6pt] (1.5,0.01) -- (1.5,3.99);
        \draw[densely dotted,line width=0.6pt] (2.5,0.01) -- (2.5,3.99);
        \draw[densely dotted,line width=0.6pt] (0.01,1.666) -- (3.99,1.666);
        \draw[densely dotted,line width=0.6pt] (0.01,1.333) -- (3.99,1.333);
        \filldraw (1.5,1.333) circle (2pt);
        \filldraw (2.5,1.666) circle (2pt);
      \end{tikzpicture}}
      \caption{The graph of the pattern $213$; an occurrence of this pattern
      within the permutation $42135$; and a graph focusing on the occurrence,
      showing the locations of the remaining points in the permutation.}
      \label{fig:flattenex}
    \end{figure}
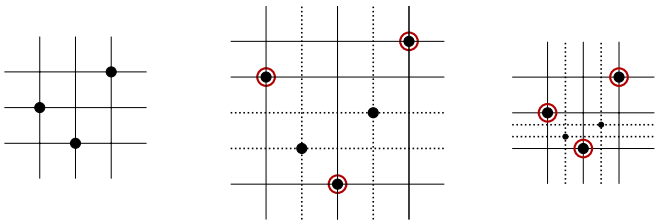

    Using the graphical representation of permutations we recall the definition
    of \emph{mesh patterns}, which are generalizations of permutation patterns
    with added restrictions. In the second subfigure of Figure~\ref{fig:flattenex},
    the points of the permutation not in the occurrence can be thought of as being
    mapped into the squares formed by the mesh.
    We can restrict when points are allowed to map
    into these squares by shading the mesh.

\begin{definition}[Br\"{a}nde\'{e}n and Claesson~\cite{branden2011mesh}]
  A \emph{mesh pattern} is an ordered pair $p=(\tau,R)$ where $\tau \in \symm_k$,
  and $R$ is a subset of the $(k+1)^2$ unit squares in
  $[0,k+1]^2$. The set $R$ is the \emph{mesh}~(shading) and squares of
  the mesh are indexed by their lower-left corners; that is, $\boks{i,j} \in
  R$ refers to the square $[i,i+1] \times [j,j+1]$. The mesh pattern
  $(\tau,R)$ is depicted graphically by drawing $\tau$ as above, and shading the
  squares of the mesh $R$. The \emph{size} of the mesh pattern $p$ is $k$ and
  denoted $|p|$.
\end{definition}

Informally, a permutation $\pi \in \symm_n$ is said to contain a mesh pattern
$p = (\tau,R)$ if $\pi$ has an occurrence of the underlying classical pattern $\tau$
such that each of the shaded squares $\boks{i,j} \in R$ correspond
to an empty region in the permutation; see Example~\ref{ex:meshinperm} and
\cite{branden2011mesh}.

\begin{example} \label{ex:meshinperm}
    Consider the mesh pattern $p = (213, \{ \boks{1,2}, \boks{2,2}, \boks{2,3}
    \})$, shown in Figure~\ref{fig:flatten2ex}. In the permutation $42135$ the
    subsequence $415$ is an occurrence of the classical pattern $213$ and the
    squares in the mesh correspond to empty regions in the permutation, as can be
    seen in the second and third subfigures in Figure~\ref{fig:flatten2ex}. Note
    that although the subsequence $215$ is an occurrence of the pattern $213$
    it does not satisfy the requirements of the mesh, as the point $3$ in the
    permutation is in the region corresponding to the shaded square
    $\boks{2,2}$.
\end{example}

\begin{figure}[ht]
  \centering
  \raisebox{0.6ex}{
  \shpattb{\pattdispscale}{3}{2,1,3}[1/2,2/2,2/3][][][][1/2/2/3, 2/2/3/3, 2/3/3/4]}
  \quad\quad
  \raisebox{0.6ex}{
    \begin{tikzpicture}[baseline=(current bounding box.center), scale=\pattdispscale]
      \useasboundingbox (0.0,-0.1) rectangle (5+1.4,5+1.1);
      \foreach \x/\y in {1/4,2/4,3/4,3/5,4/4,4/5}
        \fill[pattern color = black!65, pattern=north east lines] (\x,\y) rectangle +(1,1);
      \foreach [count=\x] \y in {4,2,1,3,5}
      \filldraw (\x,\y) circle (4pt);
      \foreach \x/\y in {1/4,3/1,5/5}
      \draw[thick,red!70!black] (\x,\y) circle (7pt);
      \draw[very thin] (1,0.01) -- (1,5.99);
      \draw[very thin] (3,0.01) -- (3,5.99);
      \draw[very thin] (5,0.01) -- (5,5.99);
      \draw[very thin] (0.01,1) -- (5.99,1);
      \draw[very thin] (0.01,4) -- (5.99,4);
      \draw[very thin] (0.01,5) -- (5.99,5);
      \draw[very thin] (3,0.01) -- (3,5.99);
      \draw[very thin] (5,0.01) -- (5,5.99);

      \draw[densely dotted, line width=0.6pt] (2,0.01) -- (2,5.99);
      \draw[densely dotted, line width=0.6pt] (4,0.01) -- (4,5.99);
      \draw[densely dotted, line width=0.6pt] (0.01,2) -- (5.99,2);
      \draw[densely dotted, line width=0.6pt] (0.01,3) -- (5.99,3);

      \foreach \xa/\ya/\xb/\yb in {1/4/3/5, 3/4/5/5, 3/5/5/6}
            {
            	\draw[line width=1pt] (\xa+0.1,\ya+0.1) rectangle (\xb-0.1,\yb-0.1);
            }
    \end{tikzpicture}}
    \quad\quad
    \raisebox{0.6ex}{
      \begin{tikzpicture}[baseline=(current bounding box.center), scale=\pattdispscale]
        \useasboundingbox (0.0,-0.1) rectangle (3+1.4,3+1.1);
        \draw[very thin] (0.01,0.01) grid (3+0.99,3+0.99);
        \foreach \x/\y in {1/2,2/2,2/3}
          \fill[pattern color = black!65, pattern=north east lines] (\x,\y) rectangle +(1,1);
        \foreach [count=\x] \y in {2,1,3}
        \filldraw (\x,\y) circle (4pt);
        \foreach \x/\y in {1/2,2/1,3/3}
        \draw[thick,red!70!black] (\x,\y) circle (7pt);
        \draw[densely dotted,line width=0.6pt] (1.5,0.01) -- (1.5,3.99);
        \draw[densely dotted,line width=0.6pt] (2.5,0.01) -- (2.5,3.99);
        \draw[densely dotted,line width=0.6pt] (0.01,1.666) -- (3.99,1.666);
        \draw[densely dotted,line width=0.6pt] (0.01,1.333) -- (3.99,1.333);
        \filldraw (1.5,1.333) circle (2pt);
        \filldraw (2.5,1.666) circle (2pt);
        \foreach \xa/\ya/\xb/\yb in {1/2/2/3, 2/2/3/3, 2/3/3/4}
            {
            	\draw[line width=1pt] (\xa+0.1,\ya+0.1) rectangle (\xb-0.1,\yb-0.1);
            }
      \end{tikzpicture}}
      \caption{The graph of the mesh pattern $(213, \{ \boks{1,2}, \boks{2,2},
      \boks{2,3} \})$; an occurrence of this pattern within the permutation
      $42135$; and a graph focusing on the occurrence.}
      \label{fig:flatten2ex}
    \end{figure}
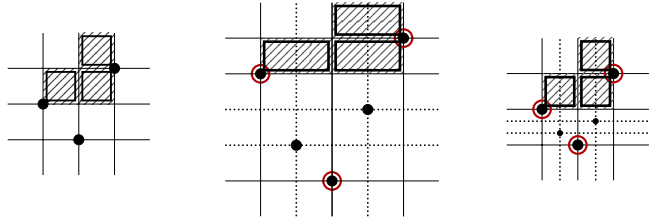

The set of permutations of size $n$ that avoid a pattern $p$ is denoted
$\avn{p}{n}$. We also define $\av{p} = \bigcup_{n = 0}^{+\infty} \avn{p}{n}$.
Its complement, the set of permutations that contain the pattern $p$, is denoted $\co{p}$.
In more generality, if $B$ is a set of patterns, we define
$\avn{B}{n}$ as the set of permutations of size $n$ that avoid all the patterns in $B$,
and $\av{B} = \bigcup_{n = 0}^{+\infty} \avn{B}{n}$. If the set $B$ is minimal, i.e.,
no $B' \subsetneq B$ with $\av{B'} = \av{B}$ exists, then $B$ is called a \emph{basis}.
When all the patterns in $B$ are
classical patterns the set $\av{B}$ is called a \emph{permutation class}.

Given two different classical patterns, $p$, $q$, it is never true that $\av{p}
= \av{q}$. This property does not hold for mesh patterns. For example,
the patterns $\textmpattern{\patttextscale}{2}{2,1}{}$ and
$\textmpattern{\patttextscale}{2}{2,1}{1/0,1/1,1/2}$ are
different mesh patterns that have the same avoiding permutations. In fact, this
is an equivalent way of stating that a permutation has an inversion if and only if
it has a descent. This leads to the following definition.

\begin{definition}
	Two patterns $p$ and $q$ are said to be \emph{coincident} if $\av{p} = \av{q}$,
	denoted $p \asymp q$.
\end{definition}

Note that if two mesh patterns $p$ and $q$ are coincident then they necessarily have the
same underlying classical patterns.
Recently Tannock and Ulfarsson~\cite[Definition 2.6]{tannock2017equivalence} defined occurrences of
mesh patterns
in mesh patterns.  Informally, the added restriction is that for a shaded square
in the occurring pattern, the corresponding squares in the containing pattern
must all be shaded; see Example~\ref{ex:meshinmesh} and \cite{tannock2017equivalence}.

\begin{figure}[ht]
  \centering
  \raisebox{0.6ex}{
  \shpattb{\pattdispscale}{3}{2,1,3}[1/2,2/2,2/3][][][][1/2/2/3, 2/2/3/3, 2/3/3/4]}
  \quad\quad
  \raisebox{0.6ex}{
    \begin{tikzpicture}[baseline=(current bounding box.center), scale=\pattdispscale]
      \useasboundingbox (0.0,-0.1) rectangle (5+1.4,5+1.1);

      \foreach \x/\y in {0/0,0/1,0/2,1/4,2/4,3/4,3/5,4/4,4/5,3/3,4/3,4/0,5/0}
        \fill[color = black!10] (\x,\y) rectangle +(1,1);
      \foreach \x/\y in {1/4,2/4,3/4,3/5,4/4,4/5}
        \fill[pattern color = black!65, pattern=north east lines] (\x,\y) rectangle +(1,1);

      \foreach [count=\x] \y in {4,2,1,3,5}
      \filldraw (\x,\y) circle (4pt);
      \foreach \x/\y in {1/4,3/1,5/5}
      \draw[thick,red!70!black] (\x,\y) circle (7pt);
      \draw[very thin] (1,0.01) -- (1,5.99);
      \draw[very thin] (3,0.01) -- (3,5.99);
      \draw[very thin] (5,0.01) -- (5,5.99);
      \draw[very thin] (0.01,1) -- (5.99,1);
      \draw[very thin] (0.01,4) -- (5.99,4);
      \draw[very thin] (0.01,5) -- (5.99,5);
      \draw[very thin] (3,0.01) -- (3,5.99);
      \draw[very thin] (5,0.01) -- (5,5.99);

      \draw[densely dotted, line width=0.6pt] (2,0.01) -- (2,5.99);
      \draw[densely dotted, line width=0.6pt] (4,0.01) -- (4,5.99);
      \draw[densely dotted, line width=0.6pt] (0.01,2) -- (5.99,2);
      \draw[densely dotted, line width=0.6pt] (0.01,3) -- (5.99,3);

      \foreach \xa/\ya/\xb/\yb in {1/4/3/5, 3/4/5/5, 3/5/5/6}
            {
            	\draw[line width=1pt] (\xa+0.1,\ya+0.1) rectangle (\xb-0.1,\yb-0.1);
            }
    \end{tikzpicture}}
    \quad\quad
    \raisebox{0.6ex}{
      \begin{tikzpicture}[baseline=(current bounding box.center), scale=\pattdispscale]
        \useasboundingbox (0.0,-0.1) rectangle (3+1.4,3+1.1);

        \foreach \x/\y in {0/0,1/2,2/2,2/3,3/0}
          \fill[color = black!10] (\x,\y) rectangle +(1,1);
        \fill[color = black!10] (0,1) rectangle +(1,0.666);
        \fill[color = black!10] (2.5,0) rectangle +(0.5,1);
        \fill[color = black!10] (2,1.666) rectangle +(1,0.3333);
        \foreach \x/\y in {1/2,2/2,2/3}
          \fill[pattern color = black!65, pattern=north east lines] (\x,\y) rectangle +(1,1);

        \draw[very thin] (0.01,0.01) grid (3+0.99,3+0.99);
        \foreach [count=\x] \y in {2,1,3}
        \filldraw (\x,\y) circle (4pt);
        \foreach \x/\y in {1/2,2/1,3/3}
        \draw[thick,red!70!black] (\x,\y) circle (7pt);
        \draw[densely dotted,line width=0.6pt] (1.5,0.01) -- (1.5,3.99);
        \draw[densely dotted,line width=0.6pt] (2.5,0.01) -- (2.5,3.99);
        \draw[densely dotted,line width=0.6pt] (0.01,1.666) -- (3.99,1.666);
        \draw[densely dotted,line width=0.6pt] (0.01,1.333) -- (3.99,1.333);
        \filldraw (1.5,1.333) circle (2pt);
        \filldraw (2.5,1.666) circle (2pt);

        \foreach \xa/\ya/\xb/\yb in {1/2/2/3, 2/2/3/3, 2/3/3/4}
            {
            	\draw[line width=1pt] (\xa+0.1,\ya+0.1) rectangle (\xb-0.1,\yb-0.1);
            }
      \end{tikzpicture}}
  \caption{The graph of the mesh pattern $(213, \{ \boks{1,2}, \boks{2,2},
    \boks{2,3} \})$; an occurrence of this pattern within the mesh pattern
  $m$, defined in Example~\ref{ex:meshinmesh};
  and a graph focusing on the occurrence.}
  \label{fig:flattenmeshmeshex}
\end{figure}
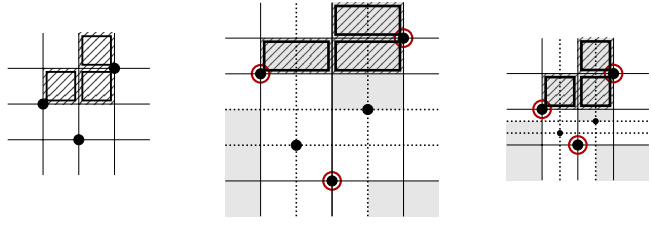

\begin{example} \label{ex:meshinmesh}
    Consider the mesh pattern $p = (213, \{ \boks{1,2}, \boks{2,2}, \boks{2,3}
      \})$, shown in Figure~\ref{fig:flattenmeshmeshex}. In the mesh pattern $m
      = (42135,$ $\{ \boks{0,0},$ $\boks{0,1},$ $\boks{0,2},$ $\boks{1,4},$
        $\boks{2,4},$ $\boks{3,3},$ $\boks{3,4},$ $\boks{3,5},$ $\boks{4,0},$
      $\boks{4,3},$ $\boks{4,4},$ $\boks{4,5},$ $\boks{5,0} \})$ the
      subsequence $415$ is an occurrence of the classical pattern $213$ and the
      squares in the mesh of $p$ correspond to regions in $m$ that are shaded
      and do not contain any points, as can be seen in the second and third
      subfigures in Figure~\ref{fig:flattenmeshmeshex}.
\end{example}

The following remark follows from the previous definitions.
\begin{remark} \label{rem:meshinmeshinperm}
    If a mesh pattern $m$ contains a mesh pattern $p$ and a permutation $\pi$
    contains $m$, then $\pi$ also contains $p$.
\end{remark}

To compare two occurrences of a mesh pattern in a permutation, or
in another mesh pattern, we need the following definition.

\begin{definition}\label{def:compare}
    Let $p=(\tau,R)$ and $q=(\sigma,T)$ be mesh patterns. If $u=\tau(u_1)$ $\dotsm
    \tau(u_k)$ and $v=\tau(v_1)\dotsm \tau(v_k)$ are occurrences of $q$ in
    $p$, then we say that
    \begin{itemize}
        \item $u$ is \emph{above} $v$ with respect to the point $(i,\sigma(i))$ if $\tau(u_i)>\tau(v_i)$,
        \item $u$ is \emph{below} $v$ with respect to the point $(i,\sigma(i))$ if $\tau(u_i)<\tau(v_i)$,
        \item $u$ is \emph{left of} $v$ with respect to the point $(i,\sigma(i))$ if $u_i<v_i$, and,
        \item $u$ is \emph{right of} $v$ with respect to the point $(i,\sigma(i))$ if $u_i>v_i$.
    \end{itemize}
\end{definition}

In Figure~\ref{fig:below} are the occurrences $u=415$ and $v=435$ of the
mesh pattern $(213,\{\boks{1,2},\boks{2,2},\boks{2,3}\})$. The occurrence $u$ is below
$v$ and $u$ is left of $v$ with respect to the point $(2,1)$.

\begin{figure}[hb]
  \centering
  \raisebox{0.6ex}{
  \shpattb{\pattdispscale}{3}{2,1,3}[1/2,2/2,2/3][][][][1/2/2/3, 2/2/3/3, 2/3/3/4]}
  \quad\quad
  \raisebox{0.6ex}{
    \begin{tikzpicture}[baseline=(current bounding box.center), scale=\pattdispscale]
      \useasboundingbox (0.0,-0.1) rectangle (5+1.4,5+1.1);
      \foreach \x/\y in {1/4,2/4,3/4,3/5,4/4,4/5}
        \fill[pattern color = black!65, pattern=north east lines] (\x,\y) rectangle +(1,1);
      \foreach [count=\x] \y in {4,2,1,3,5}
      \filldraw (\x,\y) circle (4pt);
      \foreach \x/\y in {1/4,3/1,5/5}
      \draw[thick,red!70!black] (\x,\y) circle (7pt);
      \draw[very thin] (1,0.01) -- (1,5.99);
      \draw[very thin] (3,0.01) -- (3,5.99);
      \draw[very thin] (5,0.01) -- (5,5.99);
      \draw[very thin] (0.01,1) -- (5.99,1);
      \draw[very thin] (0.01,4) -- (5.99,4);
      \draw[very thin] (0.01,5) -- (5.99,5);
      \draw[very thin] (3,0.01) -- (3,5.99);
      \draw[very thin] (5,0.01) -- (5,5.99);

      \draw[densely dotted, line width=0.6pt] (2,0.01) -- (2,5.99);
      \draw[densely dotted, line width=0.6pt] (4,0.01) -- (4,5.99);
      \draw[densely dotted, line width=0.6pt] (0.01,2) -- (5.99,2);
      \draw[densely dotted, line width=0.6pt] (0.01,3) -- (5.99,3);

      \foreach \xa/\ya/\xb/\yb in {1/4/3/5, 3/4/5/5, 3/5/5/6}
            {
            	\draw[line width=1pt] (\xa+0.1,\ya+0.1) rectangle (\xb-0.1,\yb-0.1);
            }
    \end{tikzpicture}}
    \quad\quad
    \raisebox{0.6ex}{
    \begin{tikzpicture}[baseline=(current bounding box.center), scale=\pattdispscale]
      \useasboundingbox (0.0,-0.1) rectangle (5+1.4,5+1.1);
      \foreach \x/\y in {1/4,2/4,3/4,4/4,4/5}
        \fill[pattern color = black!65, pattern=north east lines] (\x,\y) rectangle +(1,1);
      \foreach [count=\x] \y in {4,2,1,3,5}
      \filldraw (\x,\y) circle (4pt);
      \foreach \x/\y in {1/4,4/3,5/5}
      \draw[thick,red!70!black] (\x,\y) circle (7pt);
      \draw[very thin] (1,0.01) -- (1,5.99);
      \draw[very thin] (3,0.01) -- (3,5.99);
      \draw[very thin] (5,0.01) -- (5,5.99);
      \draw[very thin] (0.01,1) -- (5.99,1);
      \draw[very thin] (0.01,4) -- (5.99,4);
      \draw[very thin] (0.01,5) -- (5.99,5);
      \draw[very thin] (3,0.01) -- (3,5.99);
      \draw[very thin] (5,0.01) -- (5,5.99);

      \draw[densely dotted, line width=0.6pt] (2,0.01) -- (2,5.99);
      \draw[densely dotted, line width=0.6pt] (4,0.01) -- (4,5.99);
      \draw[densely dotted, line width=0.6pt] (0.01,2) -- (5.99,2);
      \draw[densely dotted, line width=0.6pt] (0.01,3) -- (5.99,3);

        \foreach \xa/\ya/\xb/\yb in {1/4/4/5, 4/4/5/5, 4/5/5/6}
            {
            	\draw[line width=1pt] (\xa+0.1,\ya+0.1) rectangle (\xb-0.1,\yb-0.1);
            }
    \end{tikzpicture}}
      \caption{The mesh pattern $(213, \{ \boks{1,2}, \boks{2,2},
      \boks{2,3} \})$ and two occurrences of it in the permutation
      $42135$.}
      \label{fig:below}
    \end{figure}
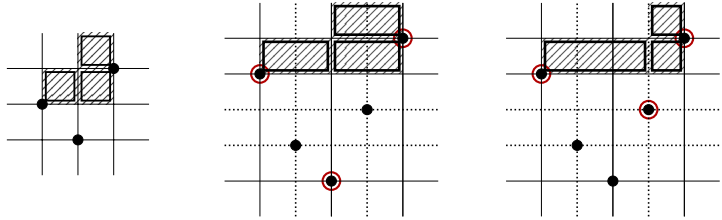

We will need to add points to mesh patterns as is formally defined in
Tannock and Ulfarsson~\cite[Definition~3.2]{tannock2017equivalence} and illustrated here in
Figure~\ref{fig:insertex}.
  \begin{figure}[ht]
    \centering
    \raisebox{0.6ex}{
    \begin{tikzpicture}[baseline=(current bounding box.center), scale=\pattdispscale]
      \useasboundingbox (0.0,-0.1) rectangle (3+1.4,3+1.1);
      \draw (0.01,0.01) grid (3+0.99,3+0.99);
      \foreach \x/\y in {1/2,2/2,2/3,0/1}
        \fill[pattern color=black!65, pattern=north east lines] (\x,\y) rectangle +(1,1);
      \foreach [count=\x] \y in {2,1,3}
        \filldraw (\x,\y) circle (3pt);
      \filldraw (2.5,1.5) circle (2pt);
      \draw[line width=0.8pt, densely dotted] (2.5,0.01) -- (2.5,3.99);
      \draw[line width=0.8pt, densely dotted] (0.01,1.5) -- (3.99,1.5);
    \end{tikzpicture}}
    \quad\quad
    \raisebox{0.6ex}{
    \begin{tikzpicture}[baseline=(current bounding box.center), scale=\pattdispscale]
      \useasboundingbox (0.0,-0.1) rectangle (4+1.4,4+1.1);
      \draw (0.01,0.01) grid (4.99,4.99);
      \foreach \x/\y in {0/1,0/2,1/3,2/3,2/4,3/3,3/4}
        \fill[pattern color=black!65, pattern=north east lines] (\x,\y) rectangle +(1,1);
      \foreach [count=\x] \y in {3,1,2,4}
        \filldraw (\x,\y) circle (3pt);
    \end{tikzpicture}}
    \caption{Adding a point to the mesh pattern
      $p = (213,\{\boks{0,1},\boks{1,2},\boks{2,2},\boks{2,3}\})$, to create the
      mesh pattern $p^\point{2,1}$.}
    \label{fig:insertex}
  \end{figure}
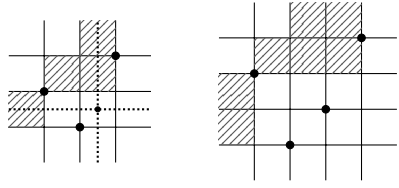
As in \cite{tannock2017equivalence} we let $p=(\tau,R)$ be a
mesh pattern such that $\boks{i,j} \not\in R$, and $p^\point{i,j} =
(\tau',T')$ be the mesh pattern with a point inserted into the square $\boks{i,j}$.
We define the following four mesh patterns, which have the same
underlying classical pattern as $p^\point{i,j}$:
    \begin{align*}
      p^\pup{i,j}    =& (\tau',T' \cup \{\boks{i,j+1},\boks{i+1,j+1} \}) \\
      p^\pdown{i,j}  =& (\tau',T' \cup \{\boks{i,j},\boks{i+1,j} \}) \\
      p^\pleft{i,j}  =& (\tau',T' \cup \{\boks{i,j},\boks{i,j+1} \}) \\
      p^\pright{i,j} =& (\tau',T' \cup \{\boks{i+1,j},\boks{i+1,j+1} \})
  \end{align*}
Informally, these patterns are obtained by placing the highest, lowest, leftmost,
or rightmost point in $\boks{i,j}$. We collect these mesh patterns in a set
\[
  p^\pall{i,j}   = \{p^\pup{i,j}, p^\pdown{i,j}, p^\pleft{i,j}, p^\pright{i,j} \}
\]
shown in Figure~\ref{fig:pointdirect} for the case $\tau = 21$ and $\boks{i,j} = \boks{1,1}$.

  \begin{figure}[hb]
    \centering
    \def\pinexsz{0.6}
    \raisebox{0.6ex}{
    \begin{tikzpicture}[baseline=(current bounding box.center), scale=.9]
      \useasboundingbox (\pinexsz,\pinexsz-0.1) rectangle (3-\pinexsz,3-\pinexsz);
      \fill[pattern color=black!65, pattern=north east lines] (1,1.5) rectangle +(1.0,0.5);
      \draw (\pinexsz,\pinexsz) grid (2.99-\pinexsz,2.99-\pinexsz);
      \foreach [count=\x] \y in {2,1}
        \filldraw (\x,\y) circle (3pt);
      \filldraw (1.5,1.5) circle (2pt);
      \draw[line width=0.8pt,densely dotted] (1.5,\pinexsz+0.01)-- (1.5,2.99-\pinexsz);
      \draw[line width=0.8pt,densely dotted] (\pinexsz+0.01,1.5)-- (2.99-\pinexsz,1.5);
    \end{tikzpicture}}
    \raisebox{0.6ex}{
    \quad\quad
    \begin{tikzpicture}[baseline=(current bounding box.center), scale=.9]
      \useasboundingbox (\pinexsz,\pinexsz-0.1) rectangle (3-\pinexsz,3-\pinexsz);
      \fill[pattern color=black!65, pattern=north east lines] (1,1) rectangle +(1.0,0.5);
      \draw (\pinexsz,\pinexsz) grid (2.99-\pinexsz,2.99-\pinexsz);
      \foreach [count=\x] \y in {2,1}
        \filldraw (\x,\y) circle (3pt);
      \filldraw (1.5,1.5) circle (2pt);
      \draw[line width=0.8pt,densely dotted] (1.5,\pinexsz+0.01)-- (1.5,2.99-\pinexsz);
      \draw[line width=0.8pt,densely dotted] (\pinexsz+0.01,1.5)-- (2.99-\pinexsz,1.5);
    \end{tikzpicture}}
    \quad\quad
    \raisebox{0.6ex}{
    \begin{tikzpicture}[baseline=(current bounding box.center), scale=.9]
      \useasboundingbox (\pinexsz,\pinexsz-0.1) rectangle (3-\pinexsz,3-\pinexsz);
      \fill[pattern color=black!65, pattern=north east lines] (1,1) rectangle +(0.5,1.0);
      \draw (\pinexsz,\pinexsz) grid (2.99-\pinexsz,2.99-\pinexsz);
      \foreach [count=\x] \y in {2,1}
        \filldraw (\x,\y) circle (3pt);
      \filldraw (1.5,1.5) circle (2pt);
      \draw[line width=0.8pt,densely dotted] (1.5,\pinexsz+0.01)-- (1.5,2.99-\pinexsz);
      \draw[line width=0.8pt,densely dotted] (\pinexsz+0.01,1.5)-- (2.99-\pinexsz,1.5);
    \end{tikzpicture}}
    \raisebox{0.6ex}{
    \quad\quad
    \begin{tikzpicture}[baseline=(current bounding box.center), scale=.9]
      \useasboundingbox (\pinexsz,\pinexsz-0.1) rectangle (3-\pinexsz,3-\pinexsz);
      \fill[pattern color=black!65, pattern=north east lines] (1.5,1) rectangle +(0.5,1.0);
      \draw (\pinexsz,\pinexsz) grid (2.99-\pinexsz,2.99-\pinexsz);
      \foreach [count=\x] \y in {2,1}
        \filldraw (\x,\y) circle (3pt);
      \filldraw (1.5,1.5) circle (2pt);
      \draw[line width=0.8pt,densely dotted] (1.5,\pinexsz+0.01)-- (1.5,2.99-\pinexsz);
      \draw[line width=0.8pt,densely dotted] (\pinexsz+0.01,1.5)-- (2.99-\pinexsz,1.5);
    \end{tikzpicture}}
    \caption{The set $p^\pall{i,j}$ for $p = 21$ and $\boks{i,j} = (1,1)$.\label{fig:pointdirect}}
  \end{figure}

We record the following easily proven statement for future reference.
\begin{remark}\label{rem:addpoint}
  Let $p=(\tau,R)$ be a mesh pattern such that $\boks{i,j} \not\in R$. A
  permutation $\pi$ that contains $p$ either contains $(\tau,R \cup
  \boks{i,j})$ or all of the patterns in $p^\pall{i,j}$.
\end{remark}

Given a mesh pattern $p$ it is clear that $p$ has an occurrence in $p$.
The indices of this occurrence, properly translated, give an occurrence of $p$
in $p^{\boks{i,j}a}$, for any $a \in \{\uparrow,\downarrow,\leftarrow,\rightarrow\}$).
We call this the \emph{trivial occurrence}. Other occurrences of $p$ in $p^{\boks{i,j}a}$
are called \emph{non-trivial}. It follows that a non-trivial occurrence contains the
inserted point $(i+1,j+1)$.

In Sections~\ref{sec:extending} and~\ref{sec:tsa} we will review existing methods
to prove that two mesh patterns are coincident and extend these methods. The
central idea will be the definition of a \emph{force} on a permutation pattern, given in
Definition~\ref{def:force}. This idea will be used to enumerate several
permutation classes in Section~\ref{sec:enum}.

\section{Extending the Shading Lemmas}\label{sec:extending}

Several proofs will make use of the following, easily proven, property.
\begin{remark} \label{rem:subsetshading}
    Let $p = (\tau, R)$ and $q = (\tau, S)$ be two mesh patterns with the
    same underlying classical pattern. If $S \subseteq R$ and a
    permutation $\pi$ contains the pattern $p$ then it contains $q$, in
    other words $\av{q} \subseteq \av{p}$.
\end{remark}

We start by recalling the Shading Lemma, from Hilmarsson et al.~\cite[Lemma 11]{shadinglemma},
and give an alternative (sketch of a) proof.

\begin{lemma}[Shading Lemma] \label{lem:shading}
Let $(\tau,R)$ be a mesh pattern of size $n$ such that $\tau(i)=j$ and the
square $\boks{i,j} \not\in R$. If all of the following conditions are satisfied:
\begin{enumerate}
  \item \label{lem:shading:req1} The square $\boks{i-1,j-1}$ is not in $R$;
  \item \label{lem:shading:req2} At most one of the squares $\boks{i,j-1}$,
    $\boks{i-1,j}$ is in $R$;
  \item \label{lem:shading:req3} If the square $\boks{\ell,j-1}$ is in $R$
    (with $\ell \not\in \{i-1,i\}$) then the square $\boks{\ell,j}$ is also in $R$;
  \item \label{lem:shading:req4} If the square $\boks{i-1,\ell}$ is in $R$
    (with $\ell \not\in \{j-1, j\}$) then the square $\boks{i,\ell}$ is also in $R$;
\end{enumerate}
then the patterns $(\tau,R)$ and $(\tau,R\cup\{\boks{i,j} \})$ are coincident.
Symmetric conditions determine if other squares neighboring the point $(i,j)$ can
be added to $R$ while preserving the coincidence of the corresponding patterns.
\end{lemma}

The conditions of the above lemma are satisfied for the mesh pattern $p = (12,$
$\{\boks{0,2},$ $\boks{1,0},$ $\boks{2,0},$ $\boks{2,1}\})$, shown on the left
below, and the square $\boks{2,2}$. The lemma therefore implies the coincidence of
$p$ with the pattern $q$ on the right.
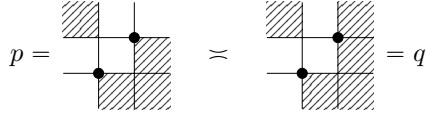
\[
  p = \shpatt{\pattdispscale}{2}{1,2}[0/2,1/0,2/0,2/1] \quad \asymp \quad
  \shpatt{\pattdispscale}{2}{1,2}[0/2,1/0,2/0,2/2,2/1] = q
\]
The argument used in the original proof relies on replacing the point corresponding to
$2$ in an occurrence of the pattern with the rightmost (or highest) point in
the region corresponding to the square $\boks{2,2}$. See Figure~\ref{fig:oldproof}, where
we have an occurrence of $p$ in the permutation $12536487$ from which we
produce an occurrence of $q$.
\begin{figure}
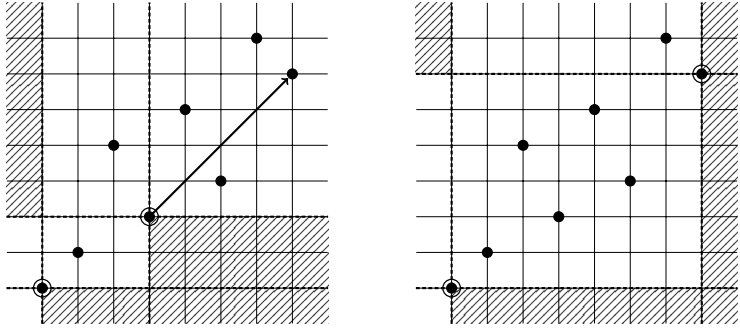

	\[
  \shpatt{\pattdispscale}{8}{1,2,5,3,6,4,8,7}[1/0,2/0,3/0,4/0,5/0,6/0,7/0,8/0,0/3,0/4,0/5,0/6,0/7,0/8,4/1,5/1,6/1,7/1,8/1,4/2,5/2,6/2,7/2,8/2][1/1,4/3][1/1,4/3][4/3/8/7]
  \quad\quad\quad
  \shpatt{\pattdispscale}{8}{1,2,5,3,6,4,8,7}[1/0,2/0,3/0,4/0,5/0,6/0,7/0,8/0,0/7,0/8,8/7,8/8,8/1,8/2,8/3,8/4,8/5,8/6][1/1,8/7][1/1,8/7][]
	\]
    \caption{Choosing the rightmost point in a region, as in the original proof of the
    Shading Lemma.\label{fig:oldproof}}
\end{figure}
The motivation for Lemma~\ref{lem:tsa1} is an alternative argument for
the coincidence of the two patterns. Consider again an occurrence of $p$ in the
same permutation, as in Figure~\ref{fig:newproof}. Out of all the occurrences of $p$ consider the occurrence
where the point corresponding to $2$ is as far to the right as possible. If the
square $\boks{2,2}$ is not empty in this occurrence then taking the lowest point in it
(for example) as a new $2$ would give us another occurrence of $p$ with the point
corresponding to $2$ further to the right, a contradiction. Therefore the square
$\boks{2,2}$ is empty and this occurrence of $p$ must also be an occurrence of $q$.
\begin{figure}
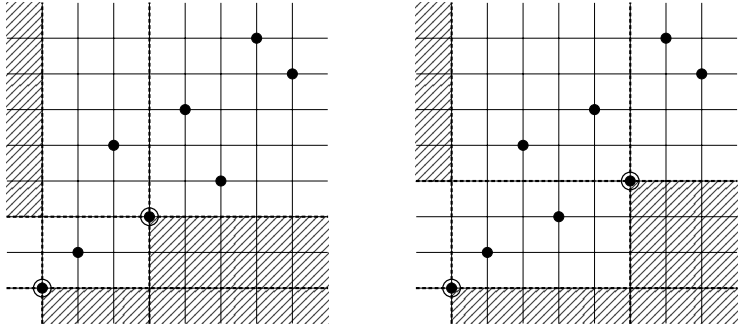

\[
	  \shpatt{\pattdispscale}{8}{1,2,5,3,6,4,8,7}[1/0,2/0,3/0,4/0,5/0,6/0,7/0,8/0,0/3,0/4,0/5,0/6,0/7,0/8,4/1,5/1,6/1,7/1,8/1,4/2,5/2,6/2,7/2,8/2][1/1,4/3][1/1,4/3][]
  \quad\quad\quad
  \shpatt{\pattdispscale}{8}{1,2,5,3,6,4,8,7}[1/0,2/0,3/0,4/0,5/0,6/0,7/0,8/0,0/4,0/5,0/6,0/7,0/8,6/1,7/1,8/1,6/2,7/2,8/2,6/3,7/3,8/3][1/1,6/4][1/1,6/4][]
\]
    \caption{If the region is not empty, we can derive a contradiction.\label{fig:newproof}}
\end{figure}
\begin{lemma}
  \label{lem:tsa1}
  Let $p = (\tau,R)$ be a mesh pattern with $\boks{i,j} \not\in R$. If a mesh
  pattern in the set $p^{\pall{i,j}}$ contains a non-trivial occurrence of $p$,
  then $p$ and $q = (\tau,R\cup\{\boks{i,j}\})$ are coincident.
\end{lemma}

\begin{proof}
  Take any $m \in p^\pall{i,j}$ such
  that $m$ has a non-trivial occurrence of $p$, which exists by the premises.
  We consider the case where $m = p^\pleft{i,j}$, as the other cases are
  symmetric to the following argument; see Figure~\ref{fig:tsa1pq}.
  In a non-trivial occurrence of $p$ in $m$, the inserted point $(i+1,j+1)$
  corresponds to some point $(k, \tau_k)$ in $p$. The point corresponding to
  $(k,\tau_k)$ in the trivial occurrence is either left or right of $(i+1,j+1)$ in $q$.
  We consider the case where it is to the left of $(i+1,j+1)$, i.e., with a
  lower index, as the other case is analogous.

  \begin{figure}[hb]
      \centering
    \begin{tikzpicture}[baseline=(current bounding box.center), scale=.8]
        \draw[rounded corners=4] (0,0) rectangle (5,5);
        \filldraw (1,1.25) circle (2pt);
        \node[anchor=north west] at (1,1.25) {$(k,\tau_k)$};
        \node[anchor=south west] at (3.5,3.5) {$\boks{i,j}$};
        \draw[gray!80!white] (2.5,0) -- (2.5,5);
        \draw[gray!80!white] (3.5,0) -- (3.5,5);
        \draw[gray!80!white] (0,2.5) -- (5,2.5);
        \draw[gray!80!white] (0,3.5) -- (5,3.5);
        \draw[very thick] (2.5,2.5) rectangle (3.5,3.5);
    \end{tikzpicture}
    \quad\quad\quad
    \begin{tikzpicture}[baseline=(current bounding box.center), scale=.8]
        \draw[rounded corners=4] (0,0) rectangle (5,5);
        \filldraw (1,1.25) circle (2pt);
        \node[anchor=north west] at (1,1.25) {$(k,\tau_k)$};
        \node[anchor=south west] at (3.5,3.5) {$\boks{i,j}$};
        \draw[gray!80!white] (2.5,0) -- (2.5,5);
        \draw[gray!80!white] (3.5,0) -- (3.5,5);
        \draw[gray!80!white] (0,2.5) -- (5,2.5);
        \draw[gray!80!white] (0,3.5) -- (5,3.5);
        \draw[very thick] (2.5,2.5) rectangle (3.5,3.5);
        \fill[pattern color=black!65, pattern=north east lines] (2.5,2.5) rectangle (3,3);
        \fill[pattern color=black!65, pattern=north east lines] (2.5,3) rectangle (3,3.5);
        \filldraw (3,3) circle (2pt);
        \draw[densely dotted,line width=0.6pt] (0,3) -- (5,3);
        \draw[densely dotted,line width=0.6pt] (3,0) -- (3,5);
        \node[anchor=east,xshift=0.1] at (2.58,4) {$(i+1,j+1)$};
        \draw (2.4,4) to[out=0,in=40,distance=20pt] (3,3);
    \end{tikzpicture}
    \caption{On the left is the pattern $p$. On the right is the pattern $m = p^\pleft{i,j}$.\label{fig:tsa1pq}}
  \end{figure}

  Take any permutation $\pi$ that contains an occurrence of $p$, and consider
  the occurrence such that the point $a$ corresponding
  to $(k,\tau_k)$ has the highest possible index in $\pi$, i.e., is as far to
  the right as possible. Assume that the region in $\pi$ that corresponds to the
  square $\boks{i,j}$ in this occurrence of $p$ is non-empty.
  The leftmost point in this region, denoted by $b$,
  along with the occurrence of $p$ in $\pi$, gives us an
  occurrence of $m = p^\pleft{i,j}$ in $\pi$; see Figure~\ref{fig:tsa1pprime}.
  This implies there is an occurrence of $p$ in
  $\pi$ with $b$ corresponding to $(k,\tau_k)$. The point $b$ is further to the right, i.e., has
  a higher index than $a$, which contradicts the choice of an occurrence of $p$
  in $\pi$. Hence, our assumption that the region corresponding to $\boks{i,j}$
  was non-empty must be false. Therefore, the region is empty, and this
  occurrence of $p$ is an occurrence of $q$ as well.

  \begin{figure}[ht]
      \centering
    \begin{tikzpicture}[baseline=(current bounding box.center), scale=0.45]
        \fill[pattern color=black!65, pattern=north east lines] (3,3) rectangle (4,4);
        \fill[pattern color=black!65, pattern=north east lines] (3,4) rectangle (4,5);
        \fill[pattern color=black!65, pattern=north east lines] (3,5) rectangle (4,6);
        \fill[pattern color=black!65, pattern=north east lines] (3,6) rectangle (4,7);
        \draw (3,3) grid (6,7);
        \draw[gray!80!white] (3,-2) -- (3,9);
        \draw[gray!80!white] (-3,3) -- (7.5,3);
        \draw[gray!80!white] (6,-2) -- (6,9);
        \draw[gray!80!white] (-3,7) -- (7.5,7);
        \draw[very thick] (3,3) rectangle (6,7);
        \filldraw (4,6) circle (2pt);
        \node[anchor=south west] at (4,6) {$b$};
        \node[anchor=south] at (4.5,7) {$\boks{i,j}$};
        \draw[rounded corners=4] (0,0) rectangle (6.5,8.2);
        \filldraw (0.7,1.8) circle (2pt);
        \node[anchor=north west] at (0.5, 2.0) {$(k,\tau_k)$};
        \draw (-3,-2) rectangle (7.5,9);
        \draw (0.7,1.8) to[out=89,in=190,distance=40pt] (4,6);
    \end{tikzpicture}
    \caption{The region corresponding to the square $\boks{i,j}$ in a permutation
    $\pi$, containing a point $b$.\label{fig:tsa1pprime}}
  \end{figure}
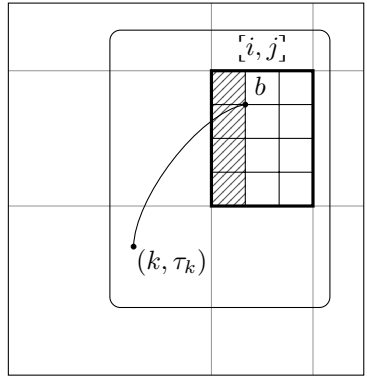
  We have shown that if a permutation contains $p$ then it contains
  $q$. By Remark~\ref{rem:subsetshading} it follows that if a permutation contains
  $q$ then it contains $p$. Therefore $p$ and $q$ are coincident.
\end{proof}

To show that the previous lemma implies any coincidence obtained with the Shading
Lemma we need the following result.

\begin{lemma}\label{lem:tsa1impliessl}
Let $p=(\tau,R)$ be a mesh pattern such that $\tau(i)=j$ and the square
$\boks{i,j} \not\in R$. If conditions
\eqref{lem:shading:req1}--\eqref{lem:shading:req4} in Lemma~\ref{lem:shading} (Shading Lemma)
are satisfied, then some $m \in p^{\pall{i,j}}$ contains a non-trivial
occurrence of $p$.
\end{lemma}

\begin{proof}
  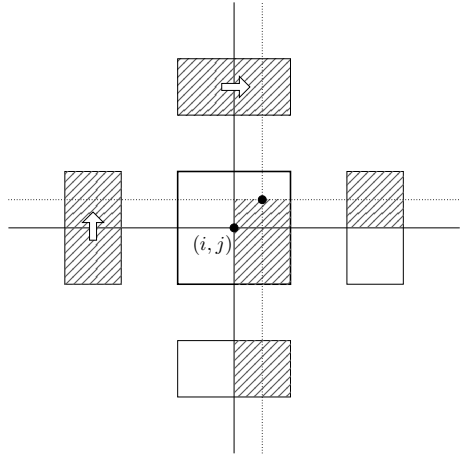
\begin{figure}[hb]
      \centering
    \scalebox{0.75}{%
      \begin{tikzpicture}[baseline=(current bounding box.center)]
          \draw (0,-4) -- (0,4);
          \draw (-4,0) -- (4,0);
          \draw[thick] (-1,-1) rectangle (1,1);
          \fill[pattern color=black!65, pattern=north east lines] (0,-1) rectangle (1,0.5);

          \draw[densely dotted] (0.5,-4) -- (0.5,4);
          \draw[densely dotted] (-4,0.5) -- (4,0.5);

          \node[anchor=north east] at (0.1,0) {$(i,j)$};

          \draw (-3,-1) rectangle (-2,1);
          \fill[pattern color=black!65, pattern=north east lines] (-3,-1) rectangle (-2,1);
          \draw (2,-1) rectangle (3,1);
          \fill[pattern color=black!65, pattern=north east lines] (2,0) rectangle (3,1);
          \node at (-2.5,0) [single arrow,shape border rotate=90,draw=black,fill=white, minimum height=10mm, scale=0.5] {};

          \draw (-1,-3) rectangle (1,-2);
          \fill[pattern color=black!65, pattern=north east lines] (-1,2) rectangle (1,3);
          \draw (-1,2) rectangle (1,3);
          \fill[pattern color=black!65, pattern=north east lines] (0,-3) rectangle (1,-2);
          \node at (0,2.5) [single arrow,shape border rotate=0,draw=black,fill=white, minimum height=10mm, scale=0.5] {};
		  \filldraw (0,0) circle (2pt);
          \filldraw (0.5,0.5) circle (2pt);

      \end{tikzpicture}}
    \caption{Obtaining a new occurrence of $p$.}
    \label{fig:shadingtsaproof}
  \end{figure}

  We will only consider the case where $\boks{i,j-1}$ is shaded, as the other
  cases are similar. Let $m = p^{\pdown{i,j}}$, as depicted in
  Figure~\ref{fig:shadingtsaproof}. By swapping out the point $(i,j)$ in the
  original occurrence of $p$ in $m$ with the inserted point, we get a non-trivial occurrence of the
  classical pattern $\tau$ in $m$. The conditions \eqref{lem:shading:req3},
  \eqref{lem:shading:req4} and the choice of $m = p^{\pdown{i,j}}$ guarantee that
  all the squares in $R$ are still shaded in this new occurrence, making it a non-trivial
  occurrence of $p$ in $m$.
\end{proof}

The previous lemma implies that the Shading Lemma (Lemma~\ref{lem:shading})
is a consequence of Lemma~\ref{lem:tsa1}. It is then natural to ask if these two
lemmas are equivalent, in the sense that they can identify exactly the same
coincidences of mesh patterns. In Example~\ref{ex:tsa1stronger} we show that
Lemma~\ref{lem:tsa1} is strictly stronger than the Shading Lemma.

\begin{example}
    \label{ex:tsa1stronger}
    Consider the mesh patterns
    \[
        p  = \shpatt{\pattdispscale}{3}{1,2,3}[0/1,1/2][][][] \quad\quad
        q = \shpatt{\pattdispscale}{3}{1,2,3}[0/0,0/1,1/2][][][]
    \]
    Then $p^{\pup{0,0}}$ is the mesh pattern
    \[
        \shpatt{\pattdispscale}{4}{1,2,3,4}[0/1,0/2,1/1,1/2,2/3][][][] \quad
    \]
    which contains a non-trivial occurrence of $p$, in the subsequence $123$. By
    Lemma~\ref{lem:tsa1} this implies that $p$ and $q$ are coincident.
    However, the conditions of Lemma~\ref{lem:shading} are not satisfied for
    the square $\boks{0,0}$ and, hence, the coincidence of these patterns does
    not follow from that lemma.
\end{example}

A previous strengthening of the Shading Lemma was given by
Claesson, Tenner, and Ulfarsson~\cite[Lemma~7.6]{simshadinglemma}:
\begin{lemma}[Simultaneous Shading Lemma]\label{lem:simshading} Let
  $p=(\tau,R)$ be a mesh pattern. Fix a subsequence $G$ of $\tau$ and, for each $g\in G$, let $U_g$
  be a square or a pair of adjacent squares that are shadeable\footnote{As defined in Tenner, Claesson, and Ulfarsson~\cite{simshadinglemma}, a square is \emph{shadeable}
  if it satisfies the conditions of Lemma~\ref{lem:shading} or any of its symmetries,
  and a pair of adjacent squares are \emph{shadeable} if they satisfy the conditions
  of Corollary 7.1 in \cite{simshadinglemma} or any of its symmetries.
  The conditions of Corollary 7.1 are those required so that the two squares can
  be shaded by two applications of Lemma~\ref{lem:shading}.}
  from $g$.  Then $p$ is coincident with
  $(\tau,R \cup S)$, where $S = \bigcup_{g \in G} U_g$.
\end{lemma}

An example of a coincidence following directly from the Simultaneous Shading
Lemma is the following:
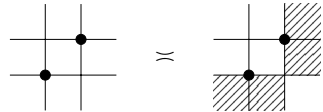
\begin{example}\label{ex:ssl_stronger_than_tsa1}
\[
  \shpatt{\pattdispscale}{2}{1,2}[][][][] \quad \asymp \quad
  \shpatt{\pattdispscale}{2}{1,2}[0/0,1/0,2/1,2/2][][][]
\]
In the pattern on the left, the square pair $\{ \boks{0,0}, \boks{1,0} \}$
is shadeable from the point $1$ and the square pair $\{ \boks{2,1}, \boks{2,2} \}$
is shadeable from the point $2$. Thus the coincidence of the patterns follows from
the Simultaneous Shading Lemma.
This, however, does not follow from Lemma~\ref{lem:tsa1}.
\end{example}

To give a common
strengthening of Lemma~\ref{lem:tsa1} and the Simultaneous Shading Lemma we need some definitions.
In Definition~\ref{def:compare} we compared occurrences of mesh patterns using a
point and a direction. We generalize this notion to include multiple points.

\begin{definition}\label{def:force}
	Given a mesh pattern $p = (\tau,R)$, with $\tau \in \symm_k$, we
define a \emph{force} on it as a tuple of pairs $F = ((\tau_{i_1}, a_1), (\tau_{i_2}, a_2),
\dots, (\tau_{i_\ell}, a_\ell))$ where $\ell \in [0, k]$, the
indices $i_j \in [1, k]$ are distinct, and $a_j \in \{\uparrow, \downarrow, \leftarrow,
\rightarrow\}$ represents the direction we are \emph{forcing} the point
$\tau_{i_j}$ in. The \emph{size} of the force $F$ is $\ell$ and denoted $|F|$.

Let $p$ be a mesh pattern with force $F$. If we have an occurrence $c =
c_1 c_2 \cdots c_k$ of $p$ in a permutation $\pi$, then for each $(\tau_{i_j},
a_j)$ we define the \emph{strength} of the point $c_{i_j}$ with respect to
the force $F$ as
\[
    \strength_F(\pi,c,i_j) = \left\{\begin{array}{ll}
                           \phantom{-}c_{i_j} & \textrm{if } a_j = \uparrow \\
                           -c_{i_j} & \textrm{if } a_j = \downarrow \\
                           -\pi^{-1}(c_{i_j}) & \textrm{if } a_j = \leftarrow \\
                           \phantom{-}\pi^{-1}(c_{i_j}) & \textrm{if } a_j = \rightarrow
                           \end{array}\right.
\]
Finally, we define the \emph{strength} of an occurrence $c$ of $p$ in a
permutation $\pi$ with respect to the force $F$ as the tuple
\[
    \strength_F(\pi, c) = (\strength_F(\pi, c, i_1), \strength_F(\pi, c, i_2), \dots, \strength_F(\pi, c, i_\ell)).
\]
An occurrence $c$ in $\pi$ is \emph{stronger} than another occurrence $c'$ in $\pi$
(\emph{with respect to $F$}) if $\strength_F(\pi, c) > \strength_F(\pi, c')$, in
the lexicographical order, otherwise it is \emph{weaker}.
\end{definition}

\begin{example}
  Consider the pattern $\tau = (1342, \emptyset)$ with force
	$F = ((2, \uparrow), (3, \downarrow))$. In the permutation $2147563$
	the subsequence $2463$ is an occurrence of $\tau$ with strength $(3,-4)$,
	while the subsequence $1563$ has strength $(3,-5)$. The first occurrence
	is stronger with respect to this force.
\end{example}


\begin{lemma}
  \label{lem:tsa2}
    Let $p = (\tau,R)$ be a mesh pattern with force $F$, and assume
    $S = \{s_1,s_2,\dots, s_k\}$ where $S \cap R = \emptyset$. If all the sets
    $p_1 = (\tau,R)^{s_1\star}, p_2 = (\tau,R\cup\{s_1\})^{s_2\star},\dots,
    p_k = (\tau,R\cup\{s_1,s_2,\dots,s_{k-1}\})^{s_k\star}$ contain an
    occurrence of $p$ that is stronger than the trivial occurrence of $p$ with respect to $F$,
    then $p$ and $q = (\tau, R \cup S)$ are coincident.
\end{lemma}

\begin{proof}
    Let $\pi$ be a permutation and let $c$ be an occurrence of $p$ in $\pi$
    which has maximal strength with respect to the force $F$. Let $i \in \{1,2,
    \dots, k\}$ and let $c'$ be a non-trivial occurrence of $p$
    in some mesh pattern in $p_{i}$ that
    is stronger than the trivial occurrence. The occurrence $c'$ gives rise
    to an occurrence of $p$ in $\pi$ which is stronger than $c$, which is a
    contradiction. Hence, in the original occurrence $c$, the region corresponding
    to the square $s_i$ is empty. Letting $i$ range from $1$ to $k$ shows that the
    regions in $\pi$ corresponding to all the squares in $S$ are empty, i.e., $c$
    is an occurrence of $q$.
  \end{proof}

In the previous lemma, the special case where $|F| = 1$ and $k = 1$ is equivalent
to Lemma~\ref{lem:tsa1}. To show that Lemma~\ref{lem:tsa2} can prove any coincidence
proven by the Simultaneous Shading Lemma we need the following result.

\begin{lemma}
    Let $p=(\tau,R)$ be a mesh pattern. Fix a subsequence $G$ of $\tau$ and, for each $g\in G$, let
    $U_g$ be a square or pair of adjacent squares that are shadeable from $g$.
    Then there exists a force $F$ such that $S = \bigcup_{g\in G} U_g$ satisfies the
    conditions of Lemma~\ref{lem:tsa2}.
\end{lemma}

\begin{proof}
   Let $k = \lvert G \rvert $. We define the force $F = ((g_1,a_1), (g_2,a_2),
   \dots, (g_k,a_k))$ as follows:\footnote{Note that in some cases $U_g$
   satisfies many of the cases, in which case we can make an arbitrary choice.}
   \[
   a_i = \left\{\begin{array}{ll}
       \uparrow & \textrm{if the square(s) $U_{g_i}$ are north of $g_i$}  \\
       \downarrow & \textrm{if the square(s) $U_{g_i}$ are south of $g_i$} \\
       \leftarrow & \textrm{if the square(s) $U_{g_i}$ are west of $g_i$} \\
       \rightarrow & \textrm{if the square(s) $U_{g_i}$ are east of $g_i$}
   \end{array}\right.
   \]
   It suffices to show that for each $s_i \in S$, some mesh pattern in $(\tau,R)^{s_i\star}$
   contains a non-trivial occurrence of $p$ that is stronger than the trivial
   occurrence of $p$. Let $U_g$ be the square (or a pair of squares) corresponding to
   $s_i$.  Since $U_g$ is shadeable from $g$ there is a mesh pattern in $(\tau,R)^{s_i\star}$ that
   contains a non-trivial occurrence of $p$ that is stronger than the trivial
   occurrence of $p$.
\end{proof}

The previous lemma implies that the Simultaneous Shading Lemma
(Lemma~\ref{lem:simshading}) is a consequence of Lemma~\ref{lem:tsa2}.
Example~\ref{ex:tsa2_stronger_than_ssl} shows that Lemma~\ref{lem:tsa2} is
strictly stronger.

\begin{example}\label{ex:tsa2_stronger_than_ssl}
    Consider the two mesh patterns
    \[
      p = \shpatt{\pattdispscale}{3}{1,2,3}[0/1,0/2,0/3, 1/1,1/3, 2/2][][][] \quad\quad
      q = \shpatt{\pattdispscale}{3}{1,2,3}[0/0,0/1,0/2,0/3, 1/1,1/3, 2/2][][][]
    \]
    Then $p^\pup{0,0}$ is the mesh pattern
    \[
        \shpatt{\pattdispscale}{4}{1,2,3,4}[0/1,0/2,0/3,0/4,1/1,1/2,1/3,1/4,2/2,2/4,3/3][][][] \quad
    \]
    which contains $p$ in the subsequence $123$. This is a stronger occurrence of $p$ than
    the trivial occurrence with respect to the force $F = ((1,\downarrow))$. By
    Lemma~\ref{lem:tsa2}, $p$ and $q$ are coincident, but $p$ does not satisfy
    the conditions of Simultaneous Shading Lemma for shading $\boks{0,0}$.
\end{example}

We introduce one final strengthening of our lemmas, before presenting an
algorithm that recursively applies them. Up to now all of the lemmas
could be used to prove that two patterns are coincident, i.e., $\av{q} = \av{p}$.
The next lemma, like Remark~\ref{rem:subsetshading}, can be used to
prove results of the form $\av{q} \subseteq \av{p}$.

\begin{lemma}
  \label{lem:tsa3}
  Let $p = (\tau,R)$ be a mesh pattern with force $F$, and $q = (\tau,R')$ be
  another mesh pattern. Let $S = \{s_1,s_2,\dots, s_k\}$ where $S = R'
  \setminus R$.  If all the sets $p_1 = (\tau,R)^{s_1\star}, p_2 =
  (\tau,R\cup\{s_1\})^{s_2\star},\dots, p_k =
  (\tau,R\cup\{s_1,s_2,\dots,s_{k-1}\})^{s_k\star}$ contain an
  occurrence of $p$ that is stronger than the trivial occurrence of $p$, or an
  occurrence of a pattern that implies an occurrence of $q$,\footnote{For instance,
  a previous application of the lemma might have shown that a pattern $p_i$ contains
  a pattern with shadings that form a superset of the shadings of $q$.} then containment
  of $p$ implies containment of $q$.
\end{lemma}

\begin{proof}
  The proof is analogous to the proof of Lemma~\ref{lem:tsa2}, with one exception.
  If a $p_i$ is ever reached such that containment of $p_i$ implies the occurrence of
  $q$, then $p$ implies the occurrence of $p_i$, which in turn implies the occurrence of $q$.
\end{proof}

Note that we can prove coincidence of two patterns $p$ and $q$ by applying the
above lemma twice: Once to prove that containment of $p$ implies containment of
$q$, and then again in the other direction to prove that containment of $q$ implies
the containment of $p$.

We have already shown that the lemmas are ordered by implication as in
Figure~\ref{fig:overview}.
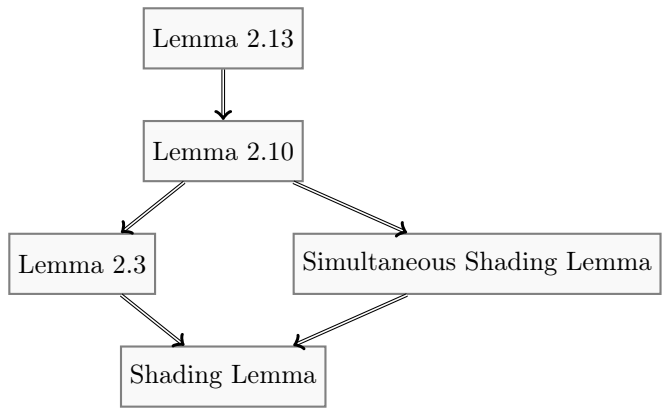
\begin{figure}[ht]
 \begin{center}
  \begin{tikzpicture}[scale = 0.75, place/.style = {rectangle,draw = black!50,fill = gray!5,thick, minimum size=0.8cm}, auto]

   \node [place] (sl) {Shading Lemma};
   \node [place] (tsa1) at ($(sl) + (-2.5, 2)$){Lemma~\ref{lem:tsa1}};
   \node [place] (ssl) at ($(tsa1) + (+7, 0)$){Simultaneous Shading Lemma};
   \node [place] (tsa2) at ($(sl) + (0, 4)$){Lemma~\ref{lem:tsa2}};
   \node [place] (tsa3) at ($(sl) + (0, 6)$){Lemma~\ref{lem:tsa3}};

   \draw [double,->] (tsa3) to (tsa2);
   \draw [double,->] (tsa2) to (ssl);
   \draw [double,->] (tsa2) to (tsa1);
   \draw [double,->] (ssl) to (sl);
   \draw [double,->] (tsa1) to (sl);
  \end{tikzpicture}
 \end{center}
 \caption{Comparison of the lemmas.}
 \label{fig:overview}
\end{figure}
To get a better idea of the power of these results, we compare them across all
mesh patterns with underlying classical patterns $1$, $12$, $123$ and
$132$.\footnote{The remaining patterns of size $2$ and $3$ are symmetries of
these patterns.} This is a total of $131.600$ mesh patterns.
Before we apply any of the lemmas we compute $\avn{p}{k}$ for $k \leq 10$. This
allows us to perform an experimental coincidence classification of these
patterns: two patterns $p$, $q$ are in the same \emph{experimental class} if
$\avn{p}{k} = \avn{q}{k}$ for $k \leq 10$. For any two patterns $p$ and $q$
in different experimental classes, the experimental classification finds a
permutation that shows that $p$ and $q$ are not coincident.  The number of
experimental classes for each underlying pattern is given in
Table~\ref{tab:coinc_classes}.
Some of these classes contain a single mesh pattern, and we call these
\emph{resolved} (res.)\ classes, since a pattern in such a class is not coincident to
any other pattern. An example of such a class is the class which contains a fully
shaded mesh pattern. Such a pattern
is avoided by every permutation but the underlying pattern itself. Therefore it is
coincident only to itself and the class is a singleton.
The remaining classes are said to be \emph{unresolved} (unr.). An unresolved class
becomes resolved when we have shown that all the patterns in the class are
coincident and therefore that the experimental class is in fact a coincidence class.

To resolve an unresolved class we start by creating a directed graph with a
vertex for each pattern in the class.  The graph is completely disconnected at first.
%
If two patterns $p$ and $q$ are shown to be coincident by the Shading Lemma,
the Simultaneous Shading Lemma, Lemma~\ref{lem:tsa1} or Lemma~\ref{lem:tsa2},
we add edges between the patterns in both directions.
If Remark~\ref{rem:subsetshading} or Lemma~\ref{lem:tsa3} shows that the
containment of $p$ implies containment of $q$ then we add a directed edge from $p$ to
$q$. If the graph becomes strongly connected (i.e., is one strong component)
then the class is resolved, as we have proven the coincidence
of all the patterns in the class.\footnote{The computations were performed on resources
provided by the Icelandic High Performance Computing Centre at the University of Iceland.}

\begin{table}[ht]
  \centering
  \begin{tabular}{l|r|r|r|r|r|r|r|r}
    Pattern & \multicolumn{2}{c|}{$1$} & \multicolumn{2}{c|}{$12$} & \multicolumn{2}{c|}{$123$} & \multicolumn{2}{c}{$132$} \\ \hline \hline
            & unr. & res. & unr. & res. & unr. & res. & unr. & res.\\ \hline
    Experimental              & $1$ & $7$ & $59$  & $161$ & $9608$  & $23908$ & $10315$ & $23035$ \\
    Shading Lemma             & $0$ & $8$ & $2$   & $218$ & $205$   & $33311$ & $183$   & $33167$ \\
    Lemma~\ref{lem:tsa1}      & $0$ & $8$ & $2$   & $218$ & $205$   & $33311$ & $183$   & $33167$ \\
    Sim. Shading Lemma        & $0$ & $8$ & $1$   & $219$ & $94$    & $33422$ & $145$   & $33205$ \\
    Lemma~\ref{lem:tsa2}      & $0$ & $8$ & $1$   & $219$ & $94$    & $33422$ & $145$   & $33205$ \\
    Lemma~\ref{lem:tsa3}      & $0$ & $8$ & $1$   & $219$ & $74$ & $33442$    & $121$   & $33229$ \\
  \end{tabular}
  \caption{The results of using the lemmas for coincidence classification of size $1$,
  $2$, and $3$ mesh patterns.\label{tab:coinc_classes}}
\end{table}


As can be seen in Table~\ref{tab:coinc_classes}, less than one percent of the
classes remain unresolved after the Shading Lemma has been applied.
We know from Lemma~\ref{lem:tsa1impliessl} and Example~\ref{ex:tsa1stronger}
that Lemma~\ref{lem:tsa1} is strictly stronger than the Shading Lemma. However,
the lines for these two lemmas in Table~\ref{tab:coinc_classes} are identical,
implying that (at least for mesh patterns with these underlying classical patterns),
this extra strength is not enough to fully resolve any more classes than the
Shading Lemma. The same phenomenon occurs for Lemma~\ref{lem:tsa2} and the
Simultaneous Shading Lemma. After Lemma~\ref{lem:tsa3} has been applied,
the remaining unresolved classes are $196$. In the next
section we further improve Lemma~\ref{lem:tsa3} and reduce
this number down to one exceptional case, which we do by hand.

\section{The Shading Algorithm}\label{sec:tsa}

We define the Shading Algorithm (Algorithm~\ref{alg:tsa1}) which iterates Lemma~\ref{lem:tsa3}.
The algorithm takes as input two mesh patterns $p = (\tau, R)$ and $q = (\tau, R')$, a
force $F$ on $\tau$ and a depth $d$. It outputs \textsc{Success} if it can show that
containment of $p$ implies containment of $q$.
%
%
The core of the algorithm is the recursive function \tsa, which takes as
input a mesh pattern $w = (\sigma, Y)$, an occurrence $c$ of $p$ in $w$ and a
depth $d$. The depth $d$ serves as a maximum recursion depth of the function.
The mesh pattern $w$ represents a state and describes an
occurrence of $p$ in an arbitrary permutation. The algorithm uses $w$ to
explore the occurrence of $p$ with maximum strength with respect to the force
$F$, which in turn is used to infer on the shadings of $p$.
Similar to the previous lemmas, the algorithm branches, depending on whether
a square in $w$ is empty or not.
In the latter case it attempts to derive a contradiction by showing that
the square can not contain a point, and is therefore empty.
We start by giving an example of what the algorithm is meant to do.

\begin{example}\label{ex:depth2}
    We want to show that an occurrence of the pattern $p = (\tau,R)$ implies an
    occurrence of the pattern $q = (\tau, R')$.
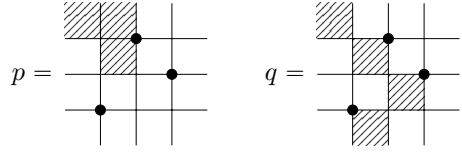
    \[
      p = \shpatt{\pattdispscale}{3}{1,3,2}[0/3, 1/2,1/3][][][] \quad\quad q = \shpatt{\pattdispscale}{3}{1,3,2}[0/3, 1/0,1/2, 2/1][][][]
    \]
    We think of $p$ as an occurrence in a permutation. The goal is to show that
    $S = R \setminus R' = \{ \boks{1,0}, \boks{2,1} \}$ can be shaded (or more precisely
    that there is an occurrence of $p$ in the permutation where $S$ is empty), or
    there is an occurrence of $q$. We choose the force $F
    = ((1,\rightarrow))$. Assume the occurrence of $p$ in the permutation maximizes
    the strength with respect to the force $F$.
    Consider the case when $\boks{1,0}$ is not empty, and add the rightmost point in
    that square to $p$.
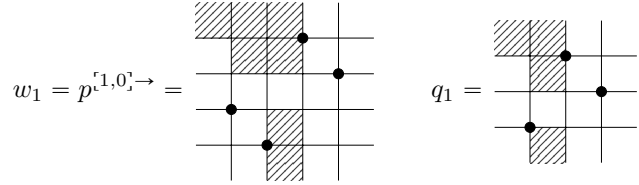
    \[
      w_1 = p^\pright{1,0} = \shpatt{\pattdispscale}{4}{2,1,4,3}[0/4, 1/3,1/4, 2/0,2/1,2/3,2/4][][][] \quad\quad q_1 = \shpatt{\pattdispscale}{3}{1,3,2}[0/3, 1/0,1/2,1/3][][][]
    \]
    Consider the subsequence at indices $234$ in $w_1$. This is an occurrence of $q_1$,
    which implies (since it has a superset of shadings) an occurrence of $p$, which is
    stronger than the original occurrence of $p$ with respect
    to $F$. This is a contradiction. Thus $\boks{1,0}$ in $p$ must have been empty.
    Consider the case when $\boks{2,1}$ is not empty in $p$,
    and add the leftmost point in that square to $p$.
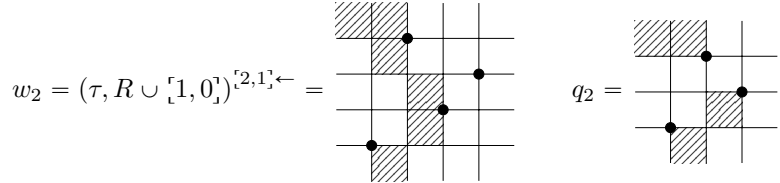
    \[
      w_2 = {(\tau,R \cup \boks{1,0})}^\pleft{2,1} = \shpatt{\pattdispscale}{4}{1,4,2,3}[0/4, 1/0,1/3,1/4, 2/1,2/2][][][] \quad\quad q_2 = \shpatt{\pattdispscale}{3}{1,3,2}[0/3, 1/0,1/3, 2/1][][][]
    \]
    We take the subsequence at indices $123$ in $w_2$. This is an occurrence of $q_2$.
    The square $\boks{1,2}$ is not shaded, and corresponds to $\boks{1,2}$ in
    $w_2$. Let us consider the rightmost point in $\boks{1,2}$ in $w_2$, giving:
    \[
      w_3 = w_2^\pright{1,2} = \shpatt{\pattdispscale}{5}{1,3,5,2,4}[0/5, 1/0,1/4,1/5, 2/0,2/2,2/3,2/4,2/5, 3/1,3/2,3/3][][][] \quad\quad q_3 = \shpatt{\pattdispscale}{3}{1,3,2}[0/3, 1/1,1/2,1/3][][][]
    \]
    Here we take the subsequence at indices $235$ in $w_3$ which is an occurrence of $q_3$,
    which implies a stronger occurrence of $p$ with respect to
    $F$. Hence, $\boks{1,2}$ in $w_2$ is empty, which implies that $\boks{1,2}$
    is empty in $q_2$. Therefore $q_2$ is a stronger occurrence of
    $p$ with respect to $F$, thus $\boks{2,1}$ is empty in $p$. The original
    occurrence of $p$ in the permutation is therefore an occurrence of $q$.
\end{example}

\begin{algorithm}[hp]
  \begin{algorithmic}[1]
    \State \textbf{Input:} Mesh patterns $p = (\tau,R)$, $F$, $q = (\tau,R')$, $d$
    \State \textbf{Output:} \textsc{Success} or \textsc{Failure}
    \Statex \,
    \Statex Shifting an occurrence after the insertion of a point
    \Function{UpdO}{$c = c_1\cdots c_n$, $\boks{x,y}$}
      \ForAll{$c_i$}
          \State \algorithmicif\ {$c_i \geq y$} \algorithmicthen\
              $c_i' \gets c_i + 1$
              \algorithmicelse\  $c_i' \gets c_i$
      \EndFor
      \State \Return $c_1'c_2'\cdots c_n'$
    \EndFunction
    \Statex Shifting a force after the insertion of a point
    \Function{UpdF}{$F' = ((t_1,a_1),\ldots,(t_k,a_k))$, $\boks{x,y}$}
      \ForAll{$t_i$}
          \State \algorithmicif\ {$t_i \geq y$} \algorithmicthen\
              $t_i' \gets t_i + 1$
              \algorithmicelse\  $t_i' \gets t_i$
      \EndFor
      \State \Return $((t_1', a_1),\ldots,(t_k',a_k))$
    \EndFunction
    \Statex \,
    \Function{\tsa}{$w = (\sigma, Y)$, $c$, $F'$, $d$}
      \For{each occurrence $c'$ of the classical pattern $\tau$ in $w$}
        \State\label{alg:tsa1:tauconstruct} Let $T$ be the maximal shading so $c'$ is an occurrence of $(\tau,T)$ in $w$
        \If{$R$ is a subset of $T$ and $\text{strength}_F(c') > \text{strength}_F(c)$}\label{alg:tsa1:base1}
          \State \Return \textsc{Success}
        \EndIf
        \If{$R'$ is a subset of $T$}\label{alg:tsa1:base2}
          \State \Return \textsc{Success}
        \EndIf
        \If{$d > 0$}\label{alg:tsa1:rec1}
            \State $\mathrm{OK} \gets \textbf{True}$
            \State Let $S = \{s_1,s_2,\dots, s_k\}$ be the squares in $w$ that correspond to $R' \setminus T$ 
            \For{$i \gets 1\ \To{}\ k$}
                \If{for all $a \in \{\uparrow, \downarrow, \leftarrow, \rightarrow\}$,
                \Statex \hskip\algorithmicindent\hskip\algorithmicindent\hskip\algorithmicindent\hskip\algorithmicindent
                $\tsa((\sigma, Y \cup \{s_1,s_2,\dots,s_{i-1}\})^{s_i a}, \textsc{UpdO}(c,s_i), \textsc{UpdF}(F',s_i), d-1)$
                \Statex \hskip\algorithmicindent\hskip\algorithmicindent\hskip\algorithmicindent\hskip\algorithmicindent
                returns \textsc{Failure}}\label{alg:tsa1:rec2}
                  \State $\mathrm{OK} \gets \textbf{False}$
                \EndIf
            \EndFor
            \If{$\mathrm{OK}$}
              \State \Return \textsc{Success}
            \EndIf
        \EndIf
      \EndFor
      \State \Return \textsc{Failure}
    \EndFunction
    \Statex \,
    \State \Return $\tsa(p, \tau, F, d)$
    \end{algorithmic}
  \caption{The Shading Algorithm.}\label{alg:tsa1}
\end{algorithm}

Before proving our main result we prove a lemma out about the case when
the function SA returns \textsc{Success}.

\begin{lemma}\label{lem:tsa4}
  For fixed mesh patterns $p$, $q$ and a force $F$ on $p$, if the function \tsa{}
  in Algorithm~\ref{alg:tsa1} returns \textsc{Success}
  for input $w = (\sigma, Y)$, $c$ (an occurrence of $p$ in $w$), $F'$ and $d$,
  then at least one of the following is false:
  \begin{enumerate}

    \item\label{it:contra1} In the strongest occurrence of $w$ with respect to $F'$ in
      any permutation $\pi$, the occurrence $c$ of $p$ in $w$ corresponds to
      the strongest occurrence of $p$ in $\pi$ with respect to $F$.
    \item\label{it:contra2} The occurrence of $w$ does not imply an occurrence
      of $q$.
  \end{enumerate}
\end{lemma}

\begin{proof}
  We prove this by induction on the depth $d$. When $d = 0$, the base case
  of the procedure on line~\ref{alg:tsa1:base1} checks whether
  statement~\eqref{it:contra1} fails and the second base case of the procedure on
  line~\ref{alg:tsa1:base2} checks whether statement~\eqref{it:contra2} fails.
  Statement~\eqref{it:contra1} fails when an occurrence of $p$ that is stronger than $c$
  w.r.t.~$F$ is found in $w$. If we consider the strongest occurrence of $w$
  w.r.t.~$F'$ in some permutation, then this stronger occurrence of $p$
  in $w$ is also a stronger occurrence of $p$ in the permutation. The second
  statement fails when $w$ contains an occurrence of $q$. Since the algorithm
  returned \textsc{Success}, either statement~\eqref{it:contra1} or
  statement~\eqref{it:contra2} is false.

  Assume the lemma holds for $d$. We now prove the inductive case for $d + 1$.
  Since the call to \tsa{} returned \textsc{Success}, there is some
  $(\tau, T)$ constructed on line~\ref{alg:tsa1:tauconstruct} that resulted in
  \textsc{Success}. If one the two tests in line~\ref{alg:tsa1:base1} or
  line~\ref{alg:tsa1:base2} succeeds then we are done. We consider the case
  where the tests fail and the algorithm continues to line~\ref{alg:tsa1:rec1}.

  For each $i = 1,\ldots,k$ we claim that if there is some $a \in
  \{\uparrow,\downarrow,\leftarrow,\rightarrow\}$ such that $\tsa({(\sigma, Y
  \cup \{s_1,\ldots,s_{i-1}\})}^{s_i a}, \textsc{UpdO}(c, s_i), \textsc{UpdF}(F', s_i), d - 1)$
  returns \textsc{Success}, then an occurrence of $(\sigma, Y \cup
  \{s_1,\ldots,s_{i-1}\})$ implies an occurrence of $(\sigma, Y \cup
  \{s_1,\ldots,s_i\})$. Consider the case when $i = 1$. By
  Remark~\ref{rem:addpoint} this occurrence of $w = (\sigma, Y)$ implies an
  occurrence of either $(\sigma, Y \cup \{s_1\})$ or the occurrence of
  $w^{\boks{i,j}a}$.
  The first case is trivial.
  For the second case we will show that the containment of $w^{\boks{i,j}a}$ leads to the same
  conclusion. The functions \textsc{UpdO} and \textsc{UpdF} update the
  occurrence $c$ in $w$ and the force $F'$, respectively, such that they refer
  to the same points after the new point was inserted into $\boks{i,j}$.  Since
  the strongest occurrence of $p$ w.r.t.~$F$ is contained in the strongest
  occurrence of $w$ w.r.t.~$F'$, it is contained in the strongest occurrence of
  $w^{\boks{i,j}a}$ w.r.t.~$\textsc{UpdF}(F', s_1)$. But since the recursive
  call on $w^{\boks{i,j}a}$ with the updated $c$ and $F'$ results in a
  contradiction, the strongest occurrence of $w$ does not imply the strongest
  occurrence of $w^{\boks{i,j}a}$ and therefore implies the occurrence of
  $(\sigma, Y \cup \{s_1\})$. The same argument holds for every $i$ and hence
  the claim holds.

  We have shown that for every $i = 1,\ldots,k$, an occurrence of $(\sigma, Y
  \cup \{s_1,\ldots,s_{i-1}\})$ implies an occurrence of $(\sigma, Y \cup
  \{s_1, \ldots,s_{i-1},s_i\})$, and hence that $w$ implies the occurrence of
  $(\sigma, Y \cup \{s_1, \ldots,s_{i-1},s_i\})$. This pattern contains an
  occurrence of $q$, thus $w$ implies an occurrence of $q$, concluding our
  proof.
\end{proof}

We now state our main result.

\begin{theorem}\label{cor:tsasound}
  If Algorithm~\ref{alg:tsa1} returns \textsc{Success} for $p$, $q$ and some
  choice of $F$, then an occurrence of $p$ implies an occurrence of $q$.
\end{theorem}

\begin{proof}
  Algorithm~\ref{alg:tsa1} calls the function \tsa{} with $p$, $\tau$, $F$ and
  $d$. By Lemma~\ref{lem:tsa4}, the algorithm returns \textsc{Success} when
  either in the strongest occurrence of $p$ w.r.t.~$F$, the occurrence $\tau$
  in $p$ does not correspond to the strongest occurrence of $p$ w.r.t.~$F$, or
  an occurrence of $p$ implies the occurrence of $q$. Since the first case
  would be a contradiction, it must be the second case.
\end{proof}

When we run the algorithm with depth $d = 2$ we are able to automatically classify
mesh patterns of size $2$, showing that the total number of coincidence
classes is $220$. The results from running Algorithm~\ref{alg:tsa1}
with $d = 1,\ldots,6$ are given in Table~\ref{tab:finaltsares}. The
algorithm fully classifies the coincidence of mesh patterns with the underlying
pattern $123$ at $d = 4$, while two classes remain unresolved for $132$ at $d =
6$. The implementation is available on GitHub~\cite{tsagithub}, and further
description of the repository is given in Section~\ref{app:tsaimplementation}.

\begin{table}[ht]
  \centering
  \begin{tabular}{l|r|r|r|r|r|r}
    Pattern & $d=1$ & $d=2$ & $d=3$ & $d=4$ & $d=5$ & $d=6$   \\ \hline \hline
      $123$ & $74$  & $8$   & $6$   & $0$   & $0$   & $0$ \\
      $132$ & $121$ & $32$  & $13$  & $6$   & $2$   & $2$ \\
  \end{tabular}
  \caption{The number of unresolved classes after running the Shading Algorithm
  at depths $d = 1, \dotsc, 6$.\label{tab:finaltsares}}
\end{table}

\begin{figure}[hb]
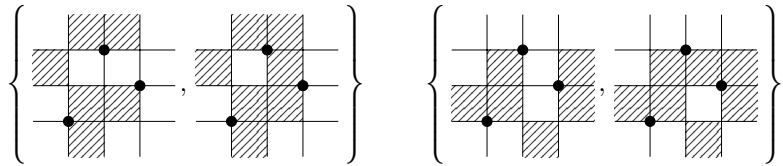

  \centering
    \[
      \left\{ \shpatt{\pattdispscale}{3}{1,3,2}[0/2,1/0,1/1,1/3,2/1,2/3][][][],
      \shpatt{\pattdispscale}{3}{1,3,2}[0/2,1/0,1/1,1/3,2/1,2/2,2/3][][][] \right\}
      \quad\quad
      \left\{ \shpatt{\pattdispscale}{3}{1,3,2}[0/1,1/1,1/2,2/0,3/1,3/2][][][],
      \shpatt{\pattdispscale}{3}{1,3,2}[0/1,1/1,1/2,2/0,2/2,3/1,3/2][][][] \right\}
    \]
  \caption{The two unresolved classes with underlying pattern $132$ after
  Algorithm~\ref{alg:tsa1} has been applied.\label{fig:twoactive}}
\end{figure}

The two remaining unresolved classes are symmetries of each other as can be seen in
Figure~\ref{fig:twoactive}.  We will therefore only prove the coincidence of
the first class since the arguments will be identical for the second class.
We start by proving that the the pattern $21$ is coincident with a \emph{decorated pattern}
(See Ulfarsson~\cite{ulfarsson2012describing} for the definition).
\begin{proposition}\label{prop:inversioncoinc}
  The patterns $p$ and $q$, shown below, are coincident. The second pattern is a decorated
  pattern that contains a (possibly empty) increasing sequence in the union of the
  squares $\boks{0,0}$ and $\boks{0,1}$, denoted by the diagonal line.
  \[
    p = 
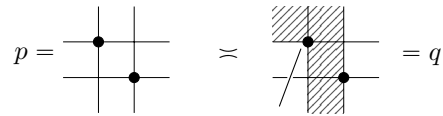
\begin{tikzpicture}[baseline=(current bounding box.center), scale=\pattdispscale]
      \useasboundingbox (0.0,-0.1) rectangle (2+1.4,2+1.1);
      \draw[very thin] (0.01,0.01) grid (2+0.99,2+0.99);
      \foreach [count=\x] \y in {2,1}
        \filldraw (\x,\y) circle (4pt);
    \end{tikzpicture}
    \quad
    \asymp
    \quad
    \begin{tikzpicture}[baseline=(current bounding box.center), scale=\pattdispscale]
      \useasboundingbox (0.0,-0.1) rectangle (2+1.4,2+1.1);
      \foreach \x/\y in {0/2,1/0,1/1,1/2}
        \fill[pattern color = black!65, pattern=north east lines] (\x,\y) rectangle +(1,1);
      \draw[very thin] (0.01,0.01) grid (2+0.99,2+0.99);
      \foreach [count=\x] \y in {2,1}
        \filldraw (\x,\y) circle (4pt);

      \draw[white, line width=2pt] (0.2,0.2) -- (0.8,1.8);
      \draw[thin, line cap=round] (0.2,0.2) -- (0.8,1.8);
    \end{tikzpicture} = q
  \]
\end{proposition}

\begin{proof}
  Let $\pi$ be a permutation. If $\pi$ contains $q$
  then it contains $p$. Assume that $\pi$ contains $p$. Since an
  occurrence of $p$ is an inversion, $\pi$ must have a descent.
  Consider the first descent in $\pi$, $ab$, with $a>b$.
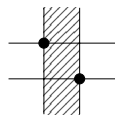
    \[
    \begin{tikzpicture}[baseline=(current bounding box.center), scale=\pattdispscale]
      \useasboundingbox (0.0,-0.1) rectangle (2+1.4,2+1.1);
      \foreach \x/\y in {1/0,1/1,1/2}
        \fill[pattern color = black!65, pattern=north east lines] (\x,\y) rectangle +(1,1);
      \draw[very thin] (0.01,0.01) grid (2+0.99,2+0.99);
      \foreach [count=\x] \y in {2,1}
        \filldraw (\x,\y) circle (4pt);
    \end{tikzpicture}
  \]
  Consider the points in the region left of the
  point $a$ in $p$, the squares $\boks{0,0}$, $\boks{0,1}$ and
  $\boks{0,2}$. If this region contains an inversion, then it must contain a
  descent. As we picked the leftmost occurrence of $p$ already, this
  region must avoid any $21$ and can only contain an increasing
  sequence. Furthermore, the square $\boks{0,2}$ must be empty, since the
  rightmost point in that region would realize a descent with its right adjacent
  point. Thus, this descent $ab$ is an occurrence of $q$.
\end{proof}

\begin{proposition}\label{prop:tsalastcase}
  The following mesh patterns are coincident.
  \[
    p = \shpatt{\pattdispscale}{3}{1,3,2}[0/2,1/0,1/1,1/3,2/1,2/3][][][]
    \asymp
    \shpatt{\pattdispscale}{3}{1,3,2}[0/2,1/0,1/1,1/3,2/1,2/2,2/3][][][] = q
  \]
\end{proposition}

\begin{proof}
  By Remark~\ref{rem:subsetshading} it suffices to show that an occurrence of $p$ implies an
  occurrence of $q$.
  Let $\pi$ be a permutation that contains $p$. Either the points in the region corresponding
  to the square $\boks{1,2}$
  form an increasing sequence~(possibly empty) or contains an inversion. We will
  show in both cases, which are depicted below, that the occurrence implies an
  occurrence of $q$.
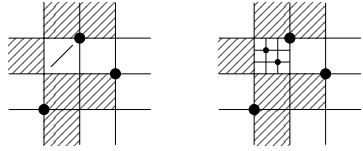
  \[
    \begin{tikzpicture}[baseline=(current bounding box.center), scale=\pattdispscale]
      \useasboundingbox (0.0,-0.1) rectangle (3+1.4,3+1.1);
      \foreach \x/\y in {0/2,1/0,1/1,1/3,2/1,2/3}
        \fill[pattern color = black!65, pattern=north east lines] (\x,\y) rectangle +(1,1);
      \draw[very thin] (0.01,0.01) grid (3+0.99,3+0.99);
      \foreach [count=\x] \y in {1,3,2}
        \filldraw (\x,\y) circle (4pt);

      \draw[line cap=round] (1.2,2.2) -- (1.8,2.8);
    \end{tikzpicture}
    \quad\quad
    \begin{tikzpicture}[baseline=(current bounding box.center), scale=\pattdispscale]
      \useasboundingbox (0.0,-0.1) rectangle (3+1.4,3+1.1);
      \foreach \x/\y in {0/2,1/0,1/1,1/3,2/1,2/3}
        \fill[pattern color = black!65, pattern=north east lines] (\x,\y) rectangle +(1,1);
      \draw[very thin] (0.01,0.01) grid (3+0.99,3+0.99);
      \foreach [count=\x] \y in {1,3,2}
        \filldraw (\x,\y) circle (4pt);
      \draw[very thin] (1.33,2) -- (1.33,3);
      \draw[very thin] (1.66,2) -- (1.66,3);
      \draw[very thin] (1,2.33) -- (2,2.33);
      \draw[very thin] (1,2.66) -- (2,2.66);
      \filldraw (1.33,2.66) circle (2pt);
      \filldraw (1.66,2.33) circle (2pt);
    \end{tikzpicture}
  \]

  In the first case, when the square $\boks{1,2}$ contains an
  increasing sequence, the square $\boks{2,2}$ either contains a point or is
  empty. If the square is empty, we are done, since this occurrence will then
  be an occurrence of $q$. Otherwise we pick the leftmost point in the square
  and place it into the occurrence of the pattern $p$.
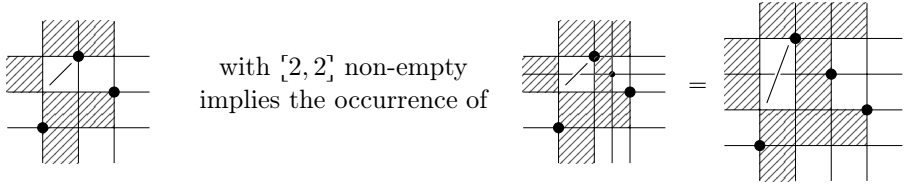
  \[
    \begin{tikzpicture}[baseline=(current bounding box.center), scale=\pattdispscale]
      \useasboundingbox (0.0,-0.1) rectangle (3+1.4,3+1.1);
      \foreach \x/\y in {0/2,1/0,1/1,1/3,2/1,2/3}
        \fill[pattern color = black!65, pattern=north east lines] (\x,\y) rectangle +(1,1);
      \draw[very thin] (0.01,0.01) grid (3+0.99,3+0.99);
      \foreach [count=\x] \y in {1,3,2}
        \filldraw (\x,\y) circle (4pt);

      \draw[line cap=round] (1.2,2.2) -- (1.8,2.8);
    \end{tikzpicture}
    \quad
    \stackanchor{ with $\boks{2,2}$ non-empty }{ implies the occurrence of }
    \quad
    \begin{tikzpicture}[baseline=(current bounding box.center), scale=\pattdispscale]
      \useasboundingbox (0.0,-0.1) rectangle (3+1.4,3+1.1);
      \foreach \x/\y in {0/2,1/0,1/1,1/3,2/1,2/3}
        \fill[pattern color = black!65, pattern=north east lines] (\x,\y) rectangle +(1,1);
      \draw[very thin] (0.01,0.01) grid (3+0.99,3+0.99);
      \foreach [count=\x] \y in {1,3,2}
        \filldraw (\x,\y) circle (4pt);
      \filldraw(2.5, 2.5) circle (2pt);
      \draw[very thin] (2.5, 0.01) -- (2.5, 3.99);
      \draw[very thin] (0.01,2.5) -- (3.99, 2.5);
      \draw[white, line width=2pt] (1.2,2.2) -- (1.8,2.8);
      \draw[line cap=round] (1.2,2.2) -- (1.8,2.8);
      \fill[pattern color = black!65, pattern=north east lines] (2,2) rectangle +(0.5,1);
    \end{tikzpicture}
    =
    \
    \begin{tikzpicture}[baseline=(current bounding box.center), scale=\pattdispscale]
      \useasboundingbox (0.0,-0.1) rectangle (4+1.4,4+1.1);
      \foreach \x/\y in {0/2,1/0,1/1,1/4,2/1,3/1,0/3,2/4,3/4,2/2,2/3}
        \fill[pattern color = black!65, pattern=north east lines] (\x,\y) rectangle +(1,1);
      \draw[very thin] (0.01,0.01) grid (4+0.99,4+0.99);
      \foreach [count=\x] \y in {1,4,3,2}
        \filldraw (\x,\y) circle (4pt);
      \draw[white, line width=2pt] (1.2,2.2) -- (1.8,3.8);
      \draw[line cap=round] (1.2,2.2) -- (1.8,3.8);
    \end{tikzpicture}
  \]
  The square $\boks{1,2}$ is either empty or
  contains a point.  If it is empty, the occurrence implies an occurrence of
  $q$ with the subsequence $143$. Otherwise,
  the square contains a point. Pick the rightmost point, which is also the
  highest point.
  
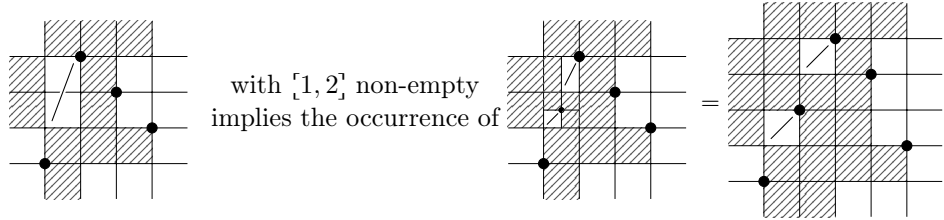
\[
    \begin{tikzpicture}[baseline=(current bounding box.center), scale=\pattdispscale]
      \useasboundingbox (0.0,-0.1) rectangle (4+1.4,4+1.1);
      \foreach \x/\y in {0/2,1/0,1/1,1/4,2/1,3/1,0/3,2/4,3/4,2/2,2/3}
        \fill[pattern color = black!65, pattern=north east lines] (\x,\y) rectangle +(1,1);
      \draw[very thin] (0.01,0.01) grid (4+0.99,4+0.99);
      \foreach [count=\x] \y in {1,4,3,2}
        \filldraw (\x,\y) circle (4pt);
      \draw[white, line width=2pt] (1.2,2.2) -- (1.8,3.8);
      \draw[line cap=round] (1.2,2.2) -- (1.8,3.8);
    \end{tikzpicture}
    \stackanchor{ with $\boks{1,2}$ non-empty }{ implies the occurrence of }
    \begin{tikzpicture}[baseline=(current bounding box.center), scale=\pattdispscale]
      \useasboundingbox (0.0,-0.1) rectangle (4+1.4,4+1.1);
      \foreach \x/\y in {0/2,1/0,1/1,1/4,2/1,3/1,0/3,2/4,3/4,2/2,2/3}
        \fill[pattern color = black!65, pattern=north east lines] (\x,\y) rectangle +(1,1);
      \draw[very thin] (0.01,0.01) grid (4+0.99,4+0.99);
      \foreach [count=\x] \y in {1,4,3,2}
        \filldraw (\x,\y) circle (4pt);
      \filldraw (1.5, 2.5) circle (2pt);

      \fill[pattern color = black!65, pattern=north east lines] (1.5,2.0) rectangle +(0.5,1);
      \fill[pattern color = black!65, pattern=north east lines] (1,2.5) rectangle +(0.5,1.5);

      \draw[white, line width=2pt] (1.1,2.1) -- (1.4,2.4);
      \draw[line cap=round] (1.1,2.1) -- (1.4,2.4);
      \draw[white, line width=2pt] (1.6,3.2) -- (1.9,3.8);
      \draw[line cap=round] (1.6,3.2) -- (1.9,3.8);

      \draw[very thin] (1.5,2.0) -- (1.5,4.0);
      \draw[very thin] (1,2.5) -- (2,2.5);
    \end{tikzpicture}
    \hspace{-3pt}
    =
    \begin{tikzpicture}[baseline=(current bounding box.center), scale=\pattdispscale]
      \useasboundingbox (0.0,-0.1) rectangle (5+1.4,5+1.1);
      \foreach \x/\y in {0/2,0/3,0/4,1/0,1/1,1/5,3/1,4/1,3/5,4/5,3/2,3/3,3/4,2/0,2/1,2/2,2/3,2/5,1/3,1/4}
        \fill[pattern color = black!65, pattern=north east lines] (\x,\y) rectangle +(1,1);
      \draw[very thin] (0.01,0.01) grid (5+0.99,5+0.99);
      \foreach [count=\x] \y in {1,3,5,4,2}
        \filldraw (\x,\y) circle (4pt);

      \draw[white, line width=2pt] (1.2,2.2) -- (1.8,2.8);
      \draw[line cap=round] (1.2,2.2) -- (1.8,2.8);
      \draw[white, line width=2pt] (2.2,4.2) -- (2.8,4.8);
      \draw[line cap=round] (2.2,4.2) -- (2.8,4.8);
    \end{tikzpicture}
  \]
  This occurrence of $p$ with the inferred points forms an occurrence of $q$,
  namely the subsequence $354$.

  Consider the case when the square $\boks{1,2}$ in $p$ contains an
  inversion. By Proposition~\ref{prop:inversioncoinc} an inversion
  is coincident with the decorated pattern in that proposition, which we place
  instead of the inversion into the occurrence of $p$.
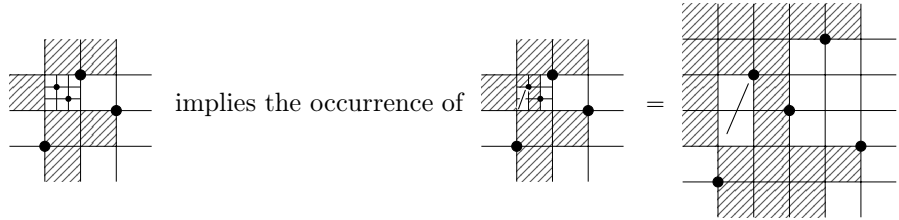
  \[
    \begin{tikzpicture}[baseline=(current bounding box.center), scale=\pattdispscale]
      \useasboundingbox (0.0,-0.1) rectangle (3+1.4,3+1.1);
      \foreach \x/\y in {0/2,1/0,1/1,1/3,2/1,2/3}
        \fill[pattern color = black!65, pattern=north east lines] (\x,\y) rectangle +(1,1);
      \draw[very thin] (0.01,0.01) grid (3+0.99,3+0.99);
      \foreach [count=\x] \y in {1,3,2}
        \filldraw (\x,\y) circle (4pt);
      \draw[very thin] (1.33,2) -- (1.33,3);
      \draw[very thin] (1.66,2) -- (1.66,3);
      \draw[very thin] (1,2.33) -- (2,2.33);
      \draw[very thin] (1,2.66) -- (2,2.66);
      \filldraw (1.33,2.66) circle (2pt);
      \filldraw (1.66,2.33) circle (2pt);
    \end{tikzpicture}
    \text{ implies the occurrence of }
    \
    \begin{tikzpicture}[baseline=(current bounding box.center), scale=\pattdispscale]
      \useasboundingbox (0.0,-0.1) rectangle (3+1.4,3+1.1);
      \fill[pattern color = black!65, pattern=north east lines] (1.33,2.0) rectangle +(0.33,1);
      \fill[pattern color = black!65, pattern=north east lines] (1,2.66) rectangle +(0.33,0.33);
      \foreach \x/\y in {0/2,1/0,1/1,1/3,2/1,2/3}
        \fill[pattern color = black!65, pattern=north east lines] (\x,\y) rectangle +(1,1);
      \draw[very thin] (0.01,0.01) grid (3+0.99,3+0.99);
      \foreach [count=\x] \y in {1,3,2}
        \filldraw (\x,\y) circle (4pt);
      \draw[very thin] (1.33,2) -- (1.33,3);
      \draw[very thin] (1.66,2) -- (1.66,3);
      \draw[very thin] (1,2.33) -- (2,2.33);
      \draw[very thin] (1,2.66) -- (2,2.66);
      \filldraw (1.33,2.66) circle (2pt);
      \filldraw (1.66,2.33) circle (2pt);

      \draw[white, line width=2pt] (1.05,2.05) -- (1.28,2.61);
      \draw[line cap=round] (1.05,2.05) -- (1.24,2.57);
    \end{tikzpicture}
    =
    \
    \begin{tikzpicture}[baseline=(current bounding box.center), scale=\pattdispscale]
      \useasboundingbox (0.0,-0.1) rectangle (5+1.4,5+1.1);

      \foreach \x/\y in {0/2,0/3,0/4,0/5,1/0,1/1,1/4,1/5,2/0,2/1,2/2,2/3,2/4,2/5,3/0,3/1,3/5,4/1,4/5}
        \fill[pattern color = black!65, pattern=north east lines] (\x,\y) rectangle +(1,1);

      \draw[very thin] (0.01,0.01) grid (5+0.99,5+0.99);
      \foreach [count=\x] \y in {1,4,3,5,2}
        \filldraw (\x,\y) circle (4pt);
      \draw[line cap=round] (1.25,2.35) -- (1.85,3.75);
    \end{tikzpicture}
  \]
  The square $\boks{1,2}$ in the occurrence is either empty or contains a
  point.  If the square is empty, the occurrence forms an occurrence of $q$
  with the subsequence $132$. We assume it contains a point and place
  the rightmost point in the square into the pattern of the occurrence.
  
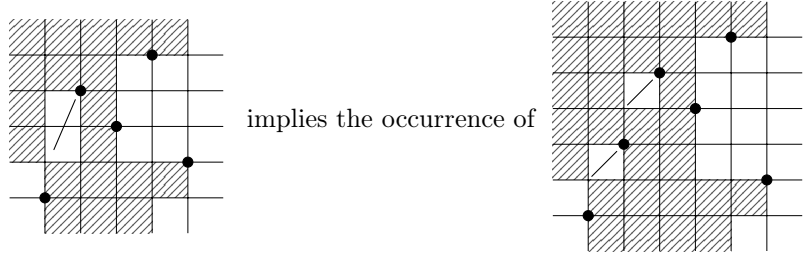
\[
    \begin{tikzpicture}[baseline=(current bounding box.center), scale=\pattdispscale]
      \useasboundingbox (0.0,-0.1) rectangle (5+1.4,5+1.1);

      \foreach \x/\y in {0/2,0/3,0/4,0/5,1/0,1/1,1/4,1/5,2/0,2/1,2/2,2/3,2/4,2/5,3/0,3/1,3/5,4/1,4/5}
        \fill[pattern color = black!65, pattern=north east lines] (\x,\y) rectangle +(1,1);

      \draw[very thin] (0.01,0.01) grid (5+0.99,5+0.99);
      \foreach [count=\x] \y in {1,4,3,5,2}
        \filldraw (\x,\y) circle (4pt);
      \draw[line cap=round] (1.25,2.35) -- (1.85,3.75);
    \end{tikzpicture}
    \
    \text{implies the occurrence of }
    \
    %
    %
    %
    %
    %
    \begin{tikzpicture}[baseline=(current bounding box.center), scale=\pattdispscale]
      \useasboundingbox (0.0,-0.1) rectangle (6+1.4,6+1.1);

      \foreach \x/\y in {0/2,0/3,0/4,0/5,0/6,1/0,1/1,1/3,1/4,1/5,1/6,2/0,2/1,2/2,2/3,2/5,2/6,3/0,3/1,3/2,3/3,3/4,3/5,3/6,4/0,4/1,4/6,5/1,5/6}
        \fill[pattern color = black!65, pattern=north east lines] (\x,\y) rectangle +(1,1);

      \draw[very thin] (0.01,0.01) grid (6+0.99,6+0.99);
      \foreach [count=\x] \y in {1,3,5,4,6,2}
        \filldraw (\x,\y) circle (4pt);
      \draw[line cap=round] (2.1,4.1) -- (2.8,4.8);
      \draw[line cap=round] (1.1,2.1) -- (1.8,2.8);
    \end{tikzpicture}
  \]
  The subsequence $354$ forms an occurrence of $q$, and we conclude that an
  occurrence of $p$ implies an occurrence of $q$.
\end{proof}

This completes the coincidence classification of mesh patterns of sizes $1$, $2$, and $3$.
There are $8$ coincidence classes of mesh patterns of size $1$. The number of classes for
longer patterns are shown in Table~\ref{tab:tsafinalres}.
\begin{table}[hb]
  \centering
  \begin{tabular}{l|r|r|r}
    Pattern & $12$ & $123$ & $132$ \\ \hline \hline
    Number of mesh patterns & $512$ & $65536$ & $65536$ \\
    Coincidence classes & $220$ & $33516$ & $33350$ \\ \hline

    Coincidence classes of size $1$ & $161$ & $23908$ & $23035$ \\
    Coincidence classes of size $2$ & $37$ & $6116$ & $6598$ \\
    Coincidence classes of size $3$ & $2$ & $132$ & $286$ \\
    Coincidence classes of size $4$ & $11$ & $1961$ & $2182$ \\
    Coincidence classes of size $5$ & $0$ & $16$ & $46$ \\
    Coincidence classes of size $6$ & $0$ & $172$ & $164$ \\
    Coincidence classes of size $7$ & $0$ & $0$ & $0$ \\
    Coincidence classes of size $\geq 8$ & $9$ & $1211$ & $1039$ \\
  \end{tabular}
  \caption{Number of coincidence classes of mesh patterns with underlying
  patterns $12$, $123$ and $132$.\label{tab:tsafinalres}}
\end{table}
The number of singleton classes in Table~\ref{tab:tsafinalres} shows how effective
the experimental classification is, since roughly a third of the patterns are
not coincident with any other mesh pattern. From the table we can also see that for $123$ and
$132$ more than $90\%$ of the coincidence classes contain at most four mesh
patterns, which greatly reduces the number of comparisons in contrast to
running the algorithm on every pair of patterns.

\section{Applications of mesh patterns and the force to enumeration}\label{sec:enum}
Knowing coincidences of mesh patterns can be used to enumerate permutation classes,
as the following example shows.

\begin{example}\label{ex:meshbinaryapplication}
  Let $B = \{1234,1243,1324,1342,1423,2314,2341,3124,4123\}$ and consider the
  permutation class $\av{B}$. Since $123$ appears as a subpattern in every pattern
  in $B$ we can write $\av{B}$ as the disjoint union
  \[
  	\av{B} = \av{123} \sqcup (\av{B} \cap \co{123}).
  \]
  The enumeration of $\av{123}$ is well known to be given by the Catalan numbers, which have
  the generating function
  \[
  	C(x) = \frac{1-\sqrt{1-4x}}{2x}.
  \]
  Therefore we only need to consider
  $\av{B} \cap \co{123}$. Using any of the lemmas in Table~\ref{tab:coinc_classes}
  we can show that $123$ is coincident with the pattern
  \[
    \shpatt{\pattdispscale}{3}{1,2,3}[0/1]
  \]
  A permutation in $\av{B}$ contains the previous pattern if and only if it contains
  the following pattern.
  \[
    \shpatt{\pattdispscale}{3}{1,2,3}[0/0,0/1,0/2,0/3,1/1,1/2,1/3,2/0,2/1,2/2,2/3,3/0,3/1,3/2,3/3]
  \]
  This is because the existence of a point in any square (except $\boks{0,1}$) would
  imply an occurrence of one of the basis elements in $B$.
  In such a
  permutation the region corresponding to the square $\boks{1,0}$ must
  avoid $12$, or else an occurrence of $3124$ would be realized. This implies
  that every permutation in $\av{B} \cap \co{123}$ has a unique occurrence of
  $123$ that is an occurrence of the pattern above. Moreover,
  any decreasing sequence of points can be placed in this region without creating
  a basis element. Every permutation in $\av{B} \cap \co{123}$ can therefore be
  constructed by starting with the permutation $123$ and placing a $12$ avoiding
  permutation in the region corresponding to $\boks{1,0}$.
  Hence, we obtain the following generating function $F_B(x)$ of the permutation class
  $\av{B}$:
  \begin{align*}
    F_{B}(x) &= F_{123}(x) + x^3 \cdot F_{12}(x)\\
             &= C(x) + \frac{x^3}{1 - x}.
  \end{align*}
\end{example}

This example motivates the following definitions.
\begin{definition}
	Let $\mathcal{C}$ be a permutation class.
	Two patterns $p$ and $q$ such that $\av{p} \cap \mathcal{C} = \av{q} \cap \mathcal{C}$
	are said to \emph{coincident with respect to $\mathcal{C}$}, denoted
	$p \overset{\mathcal{C}}{\asymp} q$.
\end{definition}

  \begin{definition}\label{def:binpatt}
  	Define $\occ{p}{\pi}$ as the number of occurrences of a pattern $p$
in a permutation $\pi$.
    A pattern $p$ is \emph{binary} if $\occ{p}{\pi} \leq
    1$ for every permutation $\pi$.
    A pattern $p$ is \emph{binary with respect to a permutation class}
    $\mathcal{C}$ if $\occ{p}{\pi} \leq 1$ for every $\pi \in
    \mathcal{C}$.
  \end{definition}

This example gives an approach to enumerating permutation classes $\mathcal{C}$:
First choose a classical pattern $p \in \mathcal{C}$	, and find a
coincident pattern $q \asymp p$ with the Shading Algorithm. Add in the shadings
implied by the basis of $\mathcal{C}$, obtaining a pattern $q' \overset{\mathcal{C}}{\asymp} p$.
If $q'$ is binary with respect to $\mathcal{C}$ it can be used to find
structural information about $\mathcal{C} \cap \co{p}$.

It is not clear what is a good choice for the pattern $p \in \mathcal{C}$, or for
the coincident pattern $q \asymp p$. In more generality, how can it be determined
when a pattern can only occur at most once in any permutation?

\begin{proposition}
    For every classical permutation pattern $p$ (except $\epsilon$, the pattern of length $0$)
    and every $i > 0$ there exists a permutation $\pi$ such that
    $\occ{p}{\pi} \geq i$.\label{prop:nobinary}
\end{proposition}

  \begin{proof}
    Let $p$ be a classical permutation pattern of size $n$.  Insert the
    element $n + 1$ to the left of the element $n$. This new permutation is of size
    $n + 1$ and has at least two occurrences of $p$. The occurrence where the element
    $n$ corresponds to the element $n$ in $p$ and the occurrence where $n + 1$
    corresponds to the element $n$ in $p$. This process can be repeated and the
    element $n + 2$ added, giving a permutation with at least one more
    occurrence of $p$ than the previous permutation. This process can be continued to
    obtain a permutation with as many occurrences of $p$ as desired.
  \end{proof}

  This proposition does not hold for mesh patterns in general. Take for
  example the mesh pattern $m = \textmpattern{\patttextscale}{1}{1}{0/0,0/1}{}$. An occurrence
  of $m$ in a permutation corresponds to the leftmost point in the
  permutation.
  Every permutation has at most one leftmost point, hence every permutation has
  at most one occurrence of $m$.
  There are also bigger and more complex mesh patterns that exhibit this
  behavior, i.e., occur exactly once or never in any permutation.

  Although Proposition~\ref{prop:nobinary} implies that no classical
  permutation pattern is binary, the pattern $m$ shows that mesh patterns can
  be binary. The pattern contains a single point that is \emph{forced}, in a
  similar manner to what was discussed in Definition~\ref{def:force}.
  Larger patterns of this type exist, such as the pattern
  $q = \textmpattern{\patttextscale}{2}{1,2}{0/0,1/0,2/0,2/0,2/1,2/2}{}$.
  The $1$ in an occurrence of $q$ in $\pi$ corresponds to $1$. Similarly $2$
  must be the rightmost point in $\pi$.
  A point in a mesh pattern of size $n$ is \emph{anchored to a boundary} if
  it is an occurrence of at least one of the mesh patterns
  $\textmpattern{\patttextscale}{1}{1}{0/1,1/1}{}$,
  $\textmpattern{\patttextscale}{1}{1}{0/0,1/0}{}$,
  $\textmpattern{\patttextscale}{1}{1}{0/0,0/1}{}$,
  $\textmpattern{\patttextscale}{1}{1}{1/0,1/1}{}$.
  Furthermore, a point is \emph{anchored}
  to another point if together they are an occurrence of at least one of the mesh patterns
  $\textmpattern{\patttextscale}{2}{1,2}{1/0,1/1,1/2}{}$,
  $\textmpattern{\patttextscale}{2}{2,1}{1/0,1/1,1/2}{}$,
  $\textmpattern{\patttextscale}{2}{1,2}{0/1,1/1,2/1}{}$,
  $\textmpattern{\patttextscale}{2}{2,1}{0/1,1/1,2/1}{}$.

  \begin{definition}
    A mesh pattern $p$ is \emph{anchored} if for every point $p_{i_1}$
    of the pattern there exists a sequence of points $p_{i_1} p_{i_2} \dots p_{i_n}$ such
    that $p_{i_\ell}$ is anchored to $p_{i_{\ell+1}}$ and $p_{i_n}$ is
    anchored to a boundary.\label{def:anchored}
  \end{definition}

  An example of an anchored pattern is given in Figure~\ref{fig:anchoredex} in
  which the point $(5,5)$ is anchored to the top boundary. The point $(4,1)$ is
  anchored to $(5,5)$, $(1,2)$ to $(4,1)$, $(3,3)$ to $(1,2)$ and $(2,4)$ to
  $(3,3)$, hence each point is anchored through a sequence of points to a boundary-anchored point.

  \begin{figure}
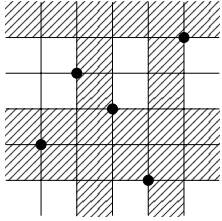

    \centering
    \shpatt{\pattdispscale}{5}{2,4,3,1,5}[0/1,0/2,0/5,1/1,1/2,1/5,2/0,2/1,2/2,2/3,2/4,2/5,3/1,3/2,3/5,4/0,4/1,4/2,4/3,4/4,4/5,5/1,5/2,5/5][][][]
    \caption{An example of an anchored pattern. The point with value $5$ is
    anchored to the boundary as the highest point and the remaining points are anchored
  to it through each other.\label{fig:anchoredex}}
  \end{figure}


  \begin{proposition}
    Every anchored mesh pattern is a binary mesh pattern.\label{prop:anchoredbin}
  \end{proposition}
  \begin{proof}

    Let $p$ be an anchored mesh pattern and let $\pi$ be a permutation that contains $p$. Any occurrence of
    $p$ in $\pi$ will use the same point in $\pi$ for a boundary-anchored point
    in $p$, e.g., if $p$ contains a point anchored to the bottom boundary
    (fully shaded bottom row), $(i,1)$ then the corresponding point in $\pi$ must be
    the lowest point in $\pi$. Any point anchored to $(i,1)$ is therefore uniquely
    determined in any occurrence of $p$ in $\pi$, since it must have an
    adjacent index, $i+1$, or value, $2$. This argument can be iterated on the points
    anchored to the points anchored to $i$ and so on. Since every point in $p$
    is anchored to a boundary-anchored point through a sequence of points, each
    point is uniquely determined in every occurrence of $p$ in $\pi$.  Hence,
    there can only be one occurrence of $p$ in $\pi$.
  \end{proof}

 The previous result implies:

  \begin{corollary}\label{cor:infbin}
    There are infinitely many binary mesh patterns.
  \end{corollary}

  Proposition~\ref{prop:anchoredbin} shows that anchored patterns are binary.
  However, non-anchored binary patterns do exist, such as the pattern in
  Figure~\ref{fig:nonanchoredmesh}.\footnote{This pattern can be shown to be binary with
  Lemma~\ref{lem:binforcemaxlen}}
  \begin{figure}[hb]
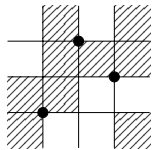

    \centering
    \shpatt{\pattdispscale}{3}{1,3,2}[0/0,0/1,1/1,1/2,1/3,2/2,3/0,3/2,3/3]
    \caption{\label{fig:nonanchoredmesh}Example of a non-anchored
    binary mesh pattern.}
  \end{figure}

  While several permutation classes can be enumerated using the method in
  Example~\ref{ex:meshbinaryapplication} we now turn our attention to a more
  powerful method, which is including a force with the pattern.


  \begin{definition}
    A \emph{forced pattern} is a tuple $(p, F)$ of a pattern
    $p$ and a force $F$.  An \emph{occurrence} of a forced pattern $(p, F)$
    in a permutation $\pi$ is an occurrence  $c$ of $p$ in $\pi$ such that
      \[\text{strength}_F(\pi, c) = \max_{\text{occ.\ }c'\text{ of }p
    \text{\ in\ }\pi}\text{strength}_F(\pi, c')\]\label{def:forceocc}
    where the force strengths are compared in the lexicographical order.
  \end{definition}

  From the definition it follows that a permutation $\pi$ contains a forced
  pattern $(p, F)$ if and only if it contains its underlying pattern $p$.
  We extend Definition~\ref{def:binpatt} to also apply to these new patterns.

  \begin{remark}
  	We note that although forced patterns are similar to anchored mesh patterns
  	in some aspects, these are different in general. We leave it to the reader
  	to show that the forced pattern pattern $(12, ((1,\leftarrow)))$ does not
  	have the same occurrences as the mesh pattern
  	$\textmpattern{\patttextscale}{2}{1,2}{0/0,0/1}{}$, or the mesh pattern
  	$\textmpattern{\patttextscale}{2}{1,2}{0/0,0/1,0/2}{}$ in a permutation.
  \end{remark}


  \begin{proposition}\label{prop:fullforce}
    A forced pattern $(p, F)$ with $|p| = |F|$ is binary.\label{prop:maxlenbin}
  \end{proposition}
  \begin{proof}
    Assume a permutation $\sigma$ has two occurrences of $p$,
    $c_1 = \sigma_{i_1}\dots \sigma_{i_n}$ and
    $c_2 = \sigma_{j_1} \dots \sigma_{j_n}$
    that have equal strength with respect to $F$. Then for every $k \in [1,n]$ we
    have either $i_k = j_k$ or $\sigma_{i_k} = \sigma_{j_k}$ from which it follows
    that the two occurrences are equal.
  \end{proof}

  Proposition~\ref{prop:maxlenbin} shows that any (classical) pattern
  can be made binary by adding a force to it. The more points that are
  forced, the fewer occurrences there are of the pattern.  If all the points are forced,
  then there is a unique occurrence of the pattern with maximum strength.

  Sometimes it may be preferable to use a force of smaller size.
  This can be important if one wants to apply the technique
  of Example~\ref{ex:binaryapplication2} to several permutation classes. Then
  for each potential pattern of length $k$ it might be too computationally
  hard to check if a force of length $k$ gives a good description of the
  subclass of the permutation class that contains the pattern.

  \begin{lemma}\label{lem:binforcemaxlen}
    Let $p$ be a pattern that is not binary with respect to the permutation
    class $\mathcal{C}$. If $p$ has size $n$, then there exists a
    permutation $\pi \in \mathcal{C}$ of size at most $2n$ such that $\occ{p}{\pi} > 1$.
  \end{lemma}
  \begin{proof}
    Since $p$ is not binary with respect to $\mathcal{C}$ there exists a permutation
    $\sigma \in \mathcal{C}$ such that $\occ{p}{\sigma} > 1$. Let $c_1 =
    \sigma_{i_1}\dots \sigma_{i_n}$ and $c_2 = \sigma_{j_1} \dots \sigma_{j_n}$
    be two distinct occurrences of $p$. Let $\pi$ be the permutation
    that is order-isomorphic to the union of the occurrences $c_1$ and $c_2$. Then
    $\pi$ has size at most $2n$, and contains at least two occurrences of $p$.
  \end{proof}

  By Lemma~\ref{lem:binforcemaxlen} we only need to check permutations up
  to size $2n$ to verify that a forced pattern of size $n$ is binary.
  Given a classical pattern, this allows us to discover, in a brute force manner, a small
  force that makes the pattern binary.
  We start with the pattern with the empty force.
  If the pattern is binary we are done, otherwise
  we pick a point and direction and add this to the force. If this forced pattern
  is now binary, we are done, else we repeat this process.
  By Proposition~\ref{prop:fullforce}, this will eventually result in a binary
  forced pattern.

\begin{example}\label{ex:binaryapplication2}
    Let $B' = \{1324, 1342, 1423, 2143, 2413, 3142\}$ and consider the
    permutation class $\av{B'}$. Similarly as in Example~\ref{ex:meshbinaryapplication}
    we get
    \[
      \av{B'} = \av{132} \sqcup (\av{B'} \cap \co{132}).
    \]
    An occurrence of $132$, in a permutation in $\av{B'}$, will be an occurrence of
    the mesh pattern
    \[
          \shpatt{\pattdispscale}{3}{1,3,2}[0/1,0/2,1/0,1/2,2/0,2/1,2/3,3/2,3/3]
    \]
    It is possible to check that there is no mesh pattern $q \overset{\av{B'}}{\asymp} 132$
    that is binary with respect to $\av{B'}$. However, $132$ is coincident to the
    forced pattern $(132, ((3, \uparrow), (1, \downarrow), (2, \downarrow)))$, which is
    binary by Proposition~\ref{prop:fullforce}.
    An occurrence of this pattern, in a permutation in $\av{B'}$, will be of the form
    \[
          \shpatt{\pattdispscale}{3}{1,3,2}[0/0,0/1,0/2,1/0,1/2,1/3,2/0,2/1,2/3,3/1,3/2,3/3]
    \]
    Because of the force the region corresponding to the square $\boks{0,3}$ avoids the
    pattern $132$. The restrictions from the basis $B'$ imply that the regions
    corresponding to the squares $\boks{1,1}$, $\boks{2,2}$, $\boks{3,0}$ avoid
    $21$, $12$, $B'$, respectively. It is easy to check that there are no other conditions
    causing these regions to be dependent.
    We therefore obtain the following equation satisfied by the generating function
    $F_{B'}(x)$ for the permutation class $\av{B'}$:
    \begin{align*}
      F_{B'}(x) &= F_{132}(x) + F_{132}(x) \cdot F_{21}(x) \cdot F_{12}(x) \cdot F_{B'}(x) \cdot x^3\\
                &= C(x) + \frac{x^3 C(x) F_{B'}(x)}{(1 - x)^2},
    \end{align*}
    where we have used the fact that $F_{132}(x) = C(x)$, the generating function of
    the Catalan numbers. Solving this equation gives
    \begin{align*}
    	F_{B'}(x) &= \frac{(1-x)^2 C(x)}{(1-x)^2 - x^3 C(x)} \\
    	          &= 1 + x + 2x^2 + 6x^3 + 18x^4 + 54x^5 + 167x^6 + 534x^7 \\
    	          &\phantom{=}+ 1755x^8 + 5896x^9 + 20167x^{10} + \dotsb.
    \end{align*}
  \end{example}

By applying the process used to enumerate $\av{B'}$ in the previous example to
non-insertion encodable\footnote{Albert, Linton, and Ru{\v{s}}kuc~\cite{albert2005insertion} studied permutation classes
with a regular language for their insertion encoding. Vatter\cite{vatter2012finding} provided
an algorithm for automatically computing the generating function of such classes.} permutation classes with bases consisting of size $4$ classical patterns
we are able to enumerate $316$ classes. No permutation class with fewer than six
patterns in its basis is successful. Detailed results are in Table~\ref{tab:run}.

\begin{table}[hb]
  \centering
  \begin{tabular}{r|r|r|r}
    Size of basis & Nr.\ of classes & Successes & \% \\ \hline \hline
    12 &   1 &   1 & 100.0 \\
    11 &  10 &  10 & 100.0 \\
    10 &  48 &  39 &  81.3 \\
     9 & 151 &  86 &  57.0 \\
     8 & 337 & 106 &  31.5 \\
     7 & 547 &  62 &  11.3 \\
     6 & 659 &  12 &   1.8 \\
     5 & 578 &   0 &   0.0
  \end{tabular}
  \caption{Number of permutation classes with bases consisting of size $4$ patterns
  that can be enumerated with the process used in Example~\ref{ex:binaryapplication2}.
  \label{tab:run}}
\end{table}

\section{Future work}\label{sec:futurework}

The experimental classification is sufficient to coincidence classify (without proof) the
mesh patterns of sizes $0$, $1$, $2$, $3$ correctly by considering the permutations
of size $1$, $3$, $5$, $10$, respectively. This leads to the following conjecture.

\begin{conjecture}\label{conj:expclas}
  The mesh patterns of size $n$ can be coincidence classified by
  experimental classification with permutations of size $(n + 1)^2 + n$.
\end{conjecture}

The intuition behind the term $(n + 1)^2 + n$ is similar to the proof of
Lemma~\ref{lem:binforcemaxlen}. Consider two mesh patterns $p$, $q$ with the
same underlying classical pattern and a permutation $\pi$ that contains $p$ but
not $q$. Then in every occurrence of $p$ in $\pi$, there must be a point in a
region that corresponds to a shaded region in $q$. Since there are $(n + 1)^2$
squares and $n$ points in a mesh pattern of size $n$, we obtain the term in
Conjecture~\ref{conj:expclas}.

One of the obvious next steps with the Shading Algorithm would be the
classification of size $4$ mesh patterns. The main issue is the number of
mesh patterns to classify since the size of the underlying pattern increases
and the number of different shadings increases to $2^{25}$. The experimental
classification becomes even more vital in this case, but the size of the
permutations to consider also increases dramatically and the task becomes
computationally impractical.

We end with two open enumerative problems:

\begin{problem}
	Given a classical pattern $p$, determine the number of anchored mesh patterns
	$(p, R)$.
\end{problem}

\begin{problem}
	Given a classical pattern $p$, determine the number of binary mesh patterns
	$(p, R)$.
\end{problem}

\bibliographystyle{acm}
\bibliography{references.bib}

\appendix
\section{Implementation of the Shading Algorithm}\label{app:tsaimplementation}
The Python implementations of the Shading Algorithm, Shading Lemma and The
Simultaneous Shading Lemma are available at Bean et al.~\cite{tsagithub}. The core of
the implementations is in the file \texttt{tsa5\_knowledge.py} under a
directory called \texttt{the\_shading\_algorithm}. The script \texttt{classify.py}
reads in experimental classes with known coincidence relations of the patterns
and calls the Shading Algorithm on pairs of patterns to decide their
coincidence. The full classification of the mesh patterns with underlying
classes $12$, $123$ and $231$ is given in the files located in the directory
\texttt{results/final\_results}.

Each result file in the \texttt{results/final\_results} of the GitHub
repository~\cite{tsagithub}, contains multiple lines, where each line represents
a coincidence class. We represent the mesh of the patterns with an integer
such that the binary representation of the integer describes the shadings.
Starting with the least significant bit, the $i$-th bit is set to $1$ if
$\boks{\lfloor i / (n + 1) \rfloor, i \mod (n + 1)}$ is shaded.

\end{document}